\documentclass[11pt]{amsart}
\usepackage[T1]{fontenc}
\usepackage{a4wide}
\usepackage{graphicx}
\usepackage{amssymb,amsmath,mathrsfs}
\usepackage{color}
\usepackage[pdftex,usenames,dvipsnames]{xcolor}
\usepackage[textsize=tiny]{todonotes}
\usepackage{verbatim}
\newtheorem{thm}{Theorem}
\newtheorem{remark}{Remark}
\newtheorem{proposition}{Proposition}
\newtheorem{example}{Example}
\newtheorem{lemma}{Lemma}
\newtheorem{definition}{Definition}
\newtheorem{corollary}{Corollary}
\newtheorem{assumption}{Assumption}

\newcommand{\eps}{\varepsilon}
\newcommand{\Wproc}{(W^k)_{k\in\mathbb{N}}} 

\newcommand{\nos}[1]{\Vert#1\Vert^2} 
\newcommand{\newt}[1]{{\color{blue}#1}} 

\newcommand{\be}[1]{\vert#1\vert} 

\newcommand{\Xe}{X}
\newcommand{\dd}{\mathrm{d}}
\newcommand{\into}{\int_{\mathcal{O}}}
\newcommand{\ska}[1]{\left( #1 \right)} 
\newcommand{\vh}{v_h}
\newcommand{\fe}[1]{\frac{\nabla #1}{ \sqrt{ \vert \nabla #1 \vert^2 +\eps^2}}} 

\newcommand{\E}[1]{\mathbb{E}\left[#1\right]} 

\begin{document}

\title[Numerical approximation of stochastic TV flow]{Numerical approximation of probabilistically weak {and strong} solutions of the stochastic {total variation} flow}
\date{\today}

\author{\v{L}ubom\'{i}r Ba\v{n}as}
\address{Department of Mathematics, Bielefeld University, 33501 Bielefeld, Germany}
\email{banas@math.uni-bielefeld.de}

\author{Martin Ondrej\'at}
\address{Institute of Information Theory and Automation, Pod Vod\'arenskou v\v e\v z\'{\i} 4, CZ-182 00, Praha 8, Czech Republic}
 \email{ondrejat@utia.cas.cz}

\thanks{The first named author was funded by the Deutsche Forschungsgemeinschaft (DFG, German Research Foundation) - SFB 1283/2 2021 - 317210226 while the second named author by the Czech Science Foundation grant no. 22-12790S}

\begin{abstract}
We propose a fully practical numerical scheme for the simulation of the stochastic total variation flow (STFV). The approximation is based on a stable time-implicit finite element space-time approximation of a regularized STVF equation. The approximation also involves a finite dimensional discretization of the noise that makes the scheme fully implementable on physical hardware. We show that the proposed numerical scheme converges to a solution that is defined in the sense of stochastic variational inequalities (SVIs). As a by product of our convergence analysis we provide a generalization of the concept of probabilistically weak solutions of stochastic partial differential equation (SPDEs) to the setting of SVIs. We also prove convergence of the numerical scheme to a probabilistically strong solution in probability if pathwise uniqueness holds. We perform numerical simulations to illustrate the behavior of the proposed numerical scheme {as well as its non-conforming variant} in the context of image denoising.
\end{abstract}

\maketitle

\section{Introduction}

We study a numerical approximation of the stochastic total variation flow (STVF) 
\begin{align}\label{TVF}
\dd X&= \mathrm{div} \left(\frac{\nabla X}{|\nabla X|}\right) \dd t -\lambda (X - g) \dd t + B(X)\dd W, &&\text{in } (0,T)\times \mathcal{O}, \nonumber\\
X & = 0 && \text{on } (0,T)\times \partial \mathcal{O}, \\
X(0)&=x^0 &&\text{in } \mathcal{O}, \nonumber
\end{align}
where $\mathcal{O} \subset \mathbb{R}^d$, $d\geq 1$ is a bounded {polyhedral} domain, $\lambda \geq 0$, $T>0$ are fixed constants and $x^0,\, g \in \mathbb{L}^2$ are given functions. 
We consider $W$ to be a cylindrical Wiener process on $\ell_2$ and a continuous mapping $B:\Bbb L^2\to\mathscr L_2(\ell_2;\Bbb L^2)$ where $\mathscr L_2$ stands for the space of Hilbert-Schmidt operators such that
\begin{enumerate}
\item[${(\mathbf B_1)}$] $\|B(h)\|_{\mathscr L_2(\ell_2;\Bbb L^2)}\le C(\|h\|+1)$ for every $h\in\Bbb L^2$,
\item[${(\mathbf B_2)}$] if $d\ge 2$, whenever $\{h_n\}$ is bounded in $\Bbb L^2$ and $h_n\to h$ a.e. in $\mathcal O$ then
$$
\|B(h_n)-B(h)\|_{\mathscr L_2(\ell_2;\Bbb L^2)}\to 0.
$$
\end{enumerate}
We also consider the weakly lower semicontinuous energy functional $\mathcal J:\Bbb L^2\to[0,\infty]$
\begin{align*}
\mathcal J(u)&:=\|\nabla u\|_{\operatorname{TV}(\mathcal{O})}+\int_{\partial\mathcal O}{|u|}\,\dd x+\frac{\lambda}{2}\into|u-g|^2\, \dd x & u&\in\Bbb L^2\cap BV(O)\,,
\\
\mathcal J(u)&:=\infty& u&\in\Bbb L^2\setminus BV(O)\,,
\end{align*}
see Lemma \ref{lem_lsc} for details.

Due to the singular character of total variation flow (\ref{TVF}), it is convenient to perform numerical simulations using a regularized problem
\begin{align}\label{reg.TVF}
 \dd \Xe&=  \mathrm{div} \left(\frac{\nabla \Xe}{\sqrt{|\nabla \Xe |^2+\eps^2}}\right)\dd t-\lambda(\Xe-g)\dd t + {B(X)}\,\dd W &&\text{in } (0,T)\times \mathcal{O}, \nonumber\\
 \Xe & = 0 && \text{on } (0,T)\times \partial \mathcal{O}, \\
 \Xe(0)&=x^0 &&\text{in } \mathcal{O}\,,  \nonumber
\end{align}
with a regularization parameter {$\eps>0$}. In the deterministic setting ($W\equiv 0$) the equation \eqref{reg.TVF} corresponds to the gradient flow of the regularized energy functional
$$
\mathcal J_\varepsilon(u):= \into \sqrt{|\nabla u|^2 +\eps^2}\, \dd x + \frac{\lambda}{2} \into |u-g|^2\, \dd x \qquad{u\in{\Bbb H^1_0}}.
$$

Convergent finite element approximation of the deterministic total variation flow (i.e., (\ref{TVF}) and (\ref{reg.TVF}) with $B(X) \equiv 0$) has been proposed
in \cite{pf03}.
In the stochastic setting, numerical approximation of probabilistically strong SVI solutions of (\ref{TVF}) with $B(X) \equiv X$ has been analyzed recently in \cite{num_stvf, brw21err, num_stvfhd} by considering the regularized
problem (\ref{reg.TVF}) withing the framework of stochastic variational inequalities, cf. \cite{br13}.
In the present work we propose a fully implementable numerical approximation of (\ref{TVF}) via the regularized problem 
(\ref{reg.TVF}): in addition to the discretization in space and time we also consider an implementable approximation of the noise term.
We show that, in the limit, the numerical solutions satisfy a stochastic variational inequality. As a consequence, we obtain an extension of the concept of stochastic variational inequalities of \cite{br13}.

{Let us compare the present work with \cite{num_stvf} where (probabilistically) strong solutions of \eqref{TVF} are constructed numerically in case the domain $\mathcal O$ is bounded, convex and with a piecewise $C^2$-smooth boundary, the equation is driven by a one-dimensional noise $W$, $B(X)=X$ and the interpolants $\overline{X}^\varepsilon_{\tau,h}$ of the numerical approximations converge to the unique solution $X$ with paths continuous in $L^2(\mathcal O)$ via the double limit
$$
\lim_{\varepsilon\to 0}\lim_{(\tau,h)\to(0,0)}\|\overline{X}^\varepsilon_{\tau,h}-X\|_{L^2(\Omega\times(0,T);L^2(\mathcal O))}=0.
$$
In the present work $\mathcal O$ is an open convex {polyhedral} domain, $B$ is a fairly general non-linearity (hence uniqueness is not expected to hold and we construct just (probabilistically weak) ``martingale'' solutions).
Furthermore, the considered noise is an infinite dimensional random walk generated by a sequence of random variables (suitable for computer simulations)
and  $\overline{X}^\varepsilon_{\tau,h}$ converge to $X$ in the joint limit as $(\varepsilon,\tau,h)\to(0,0,0)$. Our SVI solution concept is more general than the one in \cite{num_stvf} but
paths of the obtained solutions are only weakly continuous in $L^2(\mathcal O)$ and we cover the case $B(X)=X$ only in $d=1$. If, in addition, pathwise uniqueness holds then the approximations converge to a probabilistically strong solution in probability. 
We also note that the technique used for the construction of the probabilistically weak SVI solutions is straightforward, i.e., we avoid the use of martingale and Skorokhod representation theorems as in \cite{opw}.}

The paper is organized as follows. In Section~\ref{sec_num} we introduce the notation and the numerical approximation of (\ref{reg.TVF}) and  
in Section~\ref{sec_main} we state the main results of the paper {(which are proven in Sections~\ref{CoMaSo} and \ref{CoStSo})}.
In Section~\ref{sec_est} we show a priori estimate for the numerical solution.
In Section~\ref{SLCS} we present auxiliary results on compactness properties of locally convex spaces
which are used to deduce tightness properties and convergence of the numerical approximation in Section~\ref{sec_conv}.
Numerical experiments for the conforming and non-conforming finite element approximation schemes
are presented in Section~\ref{sec_numexp}. The proofs of auxiliary results are collected in the Appendix.

\section{Numerical approximation}\label{sec_num}

We denote the stanandard Lebesque and Sobolev functions spaces on $\mathcal{O}$ as
$\mathbb{L}^2:={L}^2(\mathcal{O})$, $\mathbb{L}^2_w=(\Bbb L^2,weak)$, $\mathbb{H}^1_0:={H}^1_0(\mathcal{O})$, $\mathbb{W}^{1,1} := {W}^{1,1}(\mathcal{O})$. The sets of rational and irrational numbers are denoted as $\Bbb Q$ and $\Bbb Q^\complement$, respectively. For time dependent random variables we often write $S_t(\cdot)$ instead of $S(\cdot, t)$ provided that it fits the context of presentation.


For $u \in BV(\mathcal{O})$, the gradient $\nabla u$ is a vector measure whose total variation satisfies
\begin{align}\label{total variation}
\|\nabla u\|_{\operatorname{TV}(\mathcal{O})}= \sup\left\{-\into u\, \mathrm{div}\, \mathbf{v} \mathrm{d}x;~ \mathbf{v} \in C^{\infty}_0(\mathcal{O},\mathbb{R}^d), ~\|\mathbf{v}\|_{\mathbb{L}^{\infty}}\leq 1\right\}
\end{align}
{and we define, as usual,
\begin{align*}
\|u\|_{BV(\mathcal O)}:=\|u\|_{\Bbb L^1}+\|\nabla u\|_{\operatorname{TV}(\mathcal{O})},\qquad u\in BV(\mathcal O).
\end{align*}
}

For $N\in\mathbb{N}$ we consider a discrete filtration $\mathcal{F}_\tau := \{\mathcal{F}^i_\tau \}_{i=0}^N$ on a probability space $(\Omega_\tau,\mathscr F_\tau,\Bbb P_\tau)$ 
and sequence $\{\xi^{i,j}_{\tau}\}_{i,j=1}^N$ of independent random variables 
such that 
\begin{itemize}
\item $\E{\,\xi^{i,j}_\tau}=0$,
\item $\E{\,|\xi^{i,j}_\tau|^2}=\tau$,
\item $\E{\,|\xi^{i,j}_\tau|^4}\le C\tau^2$,
\item $(\xi^{i,1}_{\tau},\dots,\xi^{i,N}_{\tau})$ is $\mathcal{F}^i_\tau$-measurable and independent of $\mathcal{F}^{i-1}_\tau$,
\end{itemize}
for every $i,j\in\{1,\dots,N\}$ and some fixed constant $C>0$ independent of $N\in\Bbb N$.
A simple and easily implementable construction of the noise that satisfies the above properties is, for instance, 
$\xi^{i,j}_\tau = \sqrt{\tau} \chi^{i,j}$ where $\{\chi^{i,j}\}_{i,j=1}^N$ are independent with $\mathbb{P}\,[\chi^{i,j} = \pm 1] = \frac{1}{2}$; as another choice, one can consider Brownian increments 
$\xi^{i,j}_\tau=\Delta_i\beta^{j} := \beta^{j}(t_i)-\beta^{j}(t_{i-1})$
of independent Brownian motions $\beta^{j}$.

Let $\Bbb V_h \subset \mathbb{H}^1_0$ be the standard finite element space of globally continuous functions 
which are piecewise linear over a quasi-uniform partition $\mathcal{T}_h$ of $\mathcal{O}$
and let $P_h:\mathbb{L}^2\rightarrow \Bbb V_h$ denote the $\mathbb{L}^2$-orthogonal projection on $\Bbb V_h$. 
We assume that the finite element space satisfies the following properties.
\begin{assumption}\label{ass_vh}
\begin{enumerate}
\item $\Bbb V_h$ is a finite-dimensional subspace of $\Bbb H^1_0$,
\item $\Bbb V_{h_2}\subseteq\Bbb V_{h_1}$ if $0<h_1<h_2$,
\item $\|P_hv\|_{\Bbb H^1_0}\le\kappa\|v\|_{\Bbb H^1_0}$ holds for every $v\in\Bbb H^1_0$ and $h>0$, for some $\kappa\in(0,\infty)$ (see \cite{bps02}),
\item $\bigcup_{h>0}\Bbb V_h$ is dense both in $\Bbb H^1_0$ and $\Bbb L^2$.
\end{enumerate}
\end{assumption}
It is well-know that the above assumption is satisfied for $\Bbb V_h$, $P_h$ see for instance \cite{book_fem}.
We note that the stability of the $\mathbb{L}^2$-projection, Assumption~\ref{ass_vh}$_{(3)}$
and the density of $\{\Bbb V_h\}_{h>0}$ in $\Bbb H^1_0$ implies that $\|\nabla v - \nabla P_{h}v\|\rightarrow 0$ as 
$h\rightarrow 0$ for every $v\in \mathbb{H}^1_0$.

We consider the following fully-discrete approximation of (\ref{reg.TVF}):
fix $N\in\mathbb{N}$, $h>0$ set $X^0 =P_hx^0$ and determine $X^i\in \mathbb{V}_h$, $i=1,\dots, N$ as the solution of
\begin{align}\label{num_tvf}
\ska{X^i-X^{i-1},\vh}  = & -\tau\left(\frac{\nabla X^i}{\sqrt{|\nabla X^i|^2+\varepsilon^2}},\nabla v_h\right) \\
&-\tau\lambda\ska{X^i - {g},\vh}+{\sum_{j=1}^N\ska{B_j(X^{i-1}),\vh}\xi^{i,j}_{\tau}} &&\forall \vh \in \mathbb{V}_h. \nonumber
\end{align}
Existence of the unique $\mathcal{F}_\tau$-adapted $\Bbb V_h$-valued solution $\{X^i\}_{i=0}^N$ can be proved analogically to \cite[Lemma 3]{num_stvf} therefore
we omit the proof. The process $X^i\equiv X^i_{\varepsilon,h}$, $i=0,\dots,N$ depends on the parameters $(\tau,h,\varepsilon)$, to simplify the notation we suppress this dependence unless it matters.

{\section{Summary of the main result{s}}\label{sec_main}

In this section, we summarize the main results of the paper.
We start with the definition of the SVI solution of \eqref{TVF}.
\begin{definition}\label{def_svi} Let $(\Omega,\mathcal F,(\mathcal F_t),\Bbb P)$ be a stochastic basis with independent $(\mathcal F_t)$-Wiener processes $\Wproc$. An adapted process $X{\in L^2([0,T]\times\Omega;\Bbb L^2)}$ with weakly continuous paths in $\Bbb L^2$ is called an SVI solution of \eqref{TVF} provided that
\begin{align}\label{SVI_def}
\frac{1}{2}\E{\|X(t)-I(t)\|^2}&+\E{\int_0^t\mathcal J(X(s))\,\dd s}\le\frac{1}{2}\|x^0-u^0\|^2
\\
&+\E{\int_0^t\mathcal J(I(s))\,\dd s}+\E{\int_0^t(G(s),X(s)-I(s))\,\dd s}\nonumber
\\
&+\frac{1}{2}\E{\int_0^t\|B(X(s))-H(s)\|_{\mathscr L_2(\ell_2,\Bbb L^2)}^2\,\dd s}\,,\nonumber
\end{align}
holds for every $t\in[0,T]$ and {for} every {test process}
\begin{equation}\label{test_process}
I(t)=u^0{-}\int_0^tG(s)\,\dd s+\sum_{j=1}^\infty\int_0^tH_j(s)\,\dd W^j,\qquad t\in[0,T]\,,
\end{equation}
that satisfies {$\Bbb P\,[I(t)\in\Bbb H^1_0]=1$ for almost every $t\in[0,T]$ and}
$$
\E{\int_0^T\|I(t)\|_{\Bbb H^1_0}\,\dd t}<\infty\,,
$$
for some $u^0\in\Bbb L^2$ and $(\mathcal F_t)$-progressively measurable processes $G$ and $H$ in $L^2([0,T]\times\Omega;\Bbb L^2)$ and $L^2([0,T]\times\Omega;\mathscr L_2(\ell_2,\Bbb L^2))$ respectively.
\end{definition}

{{
\begin{remark}
Inequality \eqref{SVI_def} implies
$$
\sup_{t\in{0,T}}\E{\|X(t)\|^2}+\E{\int_0^T\|X(t)\|_{BV(\mathcal O)}}\,\dd t<\infty.
$$
\end{remark}
}

\begin{remark}
The SVI solution in the sense of Definition \ref{def_svi} generalizes the definition of the SVI solution from \cite{br13} since the inequality \eqref{SVI_def} holds for a much larger class of test processes {introduced in \eqref{test_process}} than in \cite{br13}.
\end{remark}

We define the piecewise linear interpolant of the solution of the scheme (\ref{num_tvf}) as
\begin{equation}\label{xtaulin}
X_\tau(t) = \frac{t-t_{i-1}}{\tau}X^i + \frac{t_i-t}{\tau}X^{i-1} \quad \mathrm{for} \,\, t\in [t_{i-1}, t_i]\,,
\end{equation}
as well as the piecewise constant interpolants
\begin{subequations}\label{xtaubar}
\begin{align}
\overline{X}_\tau(t) & = X^{i} \quad \mathrm{for} \,\, t\in(t_{i-1}, t_i)\,,
\\
\underline{X}_\tau(t) & = X^{i-1} \quad \mathrm{for} \,\, t\in(t_{i-1}, t_i)\,,
\end{align}
\end{subequations}
where the dependence on $\varepsilon$, $h$ is not displayed.

\medskip

Let $\mathcal X^{(1)}$ denote the space of weakly c\`agl\`ad functions $f:[0,T]\to\Bbb L^2$ such that
$$
\int_0^T\|f(s)\|_{BV(\mathcal O)}\,ds<\infty,
$$
let $\mathcal X^{(2)}$ denote the space of weakly c\`adl\`ag functions $f:[0,T]\to\Bbb L^2$, define $\mathcal X^{(3)}$ as $C([0,T];\Bbb L^2_w)$ and equip the spaces $\mathcal X^{(1)}$, $\mathcal X^{(2)}$ and $\mathcal X^{(3)}$ with the topology of uniform convergence in $\Bbb L^2_w$.

\begin{thm}\label{main_thm} The random variables 
$$
(\overline{X}_{\eps,h,\tau},\underline{X}_{\eps,h,\tau},X_{\eps,h,\tau}):(\Omega_\tau,\mathcal F_\tau,\Bbb P_\tau)\to\mathcal X^{(1)}\times\mathcal X^{(2)}\times\mathcal X^{(3)}\,,
$$
are Borel measurable, their laws under $\Bbb P_\tau \equiv \mathbb{P}_{\eps, h, \tau}$ are tight with respect to $\eps$, $h$, $\tau$ and, moreover, every sequence $(\varepsilon_n,h_n,\tau_n)\to(0,0,0)$ has a subsequence $(\varepsilon_{n_k},h_{n_k},\tau_{n_k})$ such that laws of
$$
(\overline{X}_{\eps_{n_k},h_{n_k},\tau_{n_k}},\underline{X}_{\eps_{n_k},h_{n_k},\tau_{n_k}},X_{\eps_{n_k},h_{n_k},\tau_{n_k}})
$$
under $\Bbb P_{\tau_{n_k}}$ converge to a Radon probability measure $\nu$ on $\mathscr B(\mathcal X^{(1)}\times\mathcal X^{(2)}\times\mathcal X^{(3)})$
that satisfies
$$
\nu\big((x_1, x_2,x_3)\in \mathcal X^{(1)}\times\mathcal X^{(2)}\times\mathcal X^{(3)} :\,x_1=x_2=x_3\big)=1\,,
$$
and there exists a stochastic basis $(\Omega,\mathcal F,(\mathcal F_t),\Bbb P)$ with independent $(\mathcal F_t)$-Wiener processes $\Wproc$ and a weakly continuous $\Bbb L^2$-valued SVI solution $X$ of \eqref{TVF} in the sense of Definition \ref{def_svi} such that {$X(0)=x^0$,} $\nu$ is the law of $(X,X,X)$ on $\mathcal X^{(1)}\times\mathcal X^{(2)}\times\mathcal X^{(3)}$ {and
\begin{equation}\label{energ_eq}
\E{\sup_{t\in[0,T]}\|X(t)\|^4+\left(\int_0^T\|X(s)\|_{BV(\mathcal O)}\,\dd s\right)^2}\le C
\end{equation}
where $C$ depends only on $\|x_0\|$, $|\mathcal O|$, and $\|g\|$.
}
\end{thm}
\begin{proof}
See Corollary \ref{main_cor} and Theorem \ref{main_full_1} for the proof.
\end{proof}

\begin{remark}
Compared to the (probabilistically strong) SVI solutions in \cite{br13}, \cite{num_stvf} where the stochastic basis is given,
the SVI solution obtained in this paper is probabilistically weak
in the sense that $(\Omega,\mathcal F,(\mathcal F_t),\Bbb P) \equiv (\mathbf Z,\mathscr B(\mathbf Z),(\mathcal Z^\mu_t),\mu)$ is constructed as a part of the solution, {cf. Corollary~\ref{main_cor} and Theorem~\ref{main_full_1}}.
\end{remark}

{
\begin{remark}
If uniqueness in law holds for the SVI solution of \eqref{TVF}, cf. \cite{br13}, then the laws of
$$
(\overline{X}_{\eps,h,\tau},\underline{X}_{\eps,h,\tau},X_{\eps,h,\tau})
$$
under $\Bbb P_{\tau}$ converge to $\nu$ on $\mathscr B(\mathcal X^{(1)}\times\mathcal X^{(2)}\times\mathcal X^{(3)})$ as $(\varepsilon,h,\tau)\to(0,0,0)$, in particular, there is no need to pass to a subsequence in Theorem \ref{main_thm}.
\end{remark}

In case we work on a single stochastic basis with a given Wiener process and pathwise uniqueness holds for \eqref{TVF} then we can construct probabilistically strong solutions.

\begin{thm}\label{main_thm_2}
Let $(W^j)_{j\in\mathbb{N}}$ be independent $(\mathcal F_t)$-Wiener processes on $(\Omega,\mathcal F,(\mathcal F_t),\Bbb P)$ and let
$$
\xi_\tau^{i,j}=W^j(t_i)-W^j(t_{i-1}),\qquad t_i=i\tau.
$$
Furthremore, assume that pathwise uniqueness holds for the SVI solutions of \eqref{TVF} satisfying \eqref{energ_eq}. Then there exists an SVI solution $X$ {with respect to $(W^j)_{j\in\mathbb{N}}$} satisfying \eqref{energ_eq} such that
$$
\sup_{t\in[0,T]}|(\overline{X}_{\varepsilon,h,\tau}(t)-X(t),\varphi)|,\qquad\sup_{t\in[0,T]}|(\underline{X}_{\varepsilon,h,\tau}(t)-X(t),\varphi)|,\quad\sup_{t\in[0,T]}|(X_{\varepsilon,h,\tau}(t)-X(t),\varphi)|
$$
converge to $0$ in probability as $(\varepsilon,h,\tau)\to(0,0,0)$ for every $\varphi\in\Bbb L^2$.
\end{thm}

\begin{proof}
See Theorem \ref{strong_exist}.
\end{proof}

{
\begin{remark}
Theorem \ref{main_thm} and Theorem \ref{main_thm_2} {can be strengthened} considerably {by Lemma \ref{cont_meas}}. 
Assume that
$$
K:\mathcal X^{(1)}\times\mathcal X^{(1)}\times\mathcal X^{(2)}\times\mathcal X^{(3)}\to[0,\infty]
$$
satisfies the following property:
\begin{equation}\label{lsc_0}
K(f^0,f^1,g,h)\le\liminf_{n\to\infty}K(f^0_n,f^1_n,g_n,h_n)
\end{equation}
for any sequence $(f^0_k,f^1_k,g_k,h_k)\in\mathcal X^{(1)}\times\mathcal X^{(1)}\times\mathcal X^{(2)}\times\mathcal X^{(3)}$ converging in $\mathcal X^{(1)}\times\mathcal X^{(1)}\times\mathcal X^{(2)}\times\mathcal X^{(3)}$ to $(f^0,f^1,g,h)$ where in addition 
$$
\sup_k\int_0^T[\|f^0_k(s)\|_{BV(\mathcal O)}+\|f^1_k(s)\|_{BV(\mathcal O)}]\,\dd s<\infty\,.
$$
Then the variables $(\overline{X}_{\eps_{n_k},h_{n_k},\tau_{n_k}},\underline{X}_{\eps_{n_k},h_{n_k},\tau_{n_k}},X_{\eps_{n_k},h_{n_k},\tau_{n_k}})$
from Theorem \ref{main_thm} satisfy
\begin{equation}\label{lsc_1}
\Bbb E\left[K(X,X,X,X)\right]\le\liminf_{k\to\infty}\Bbb E_{\tau_k}\left[K(\overline{X}_{\eps_{n_k},h_{n_k},\tau_{n_k}},\overline{X}_{\eps_{n_k},h_{n_k},\tau_{n_k}},\underline{X}_{\eps_{n_k},h_{n_k},\tau_{n_k}},X_{\eps_{n_k},h_{n_k},\tau_{n_k}})\right]
\end{equation}
and the random variables $(X,\overline{X}_{\eps,h,\tau},\underline{X}_{\eps,h,\tau},X_{\eps,h,\tau})$ from Theorem \ref{main_thm_2} satisfy
\begin{equation}\label{lsc_2}
\Bbb E\left[K(X,X,X,X)\right]\le\liminf_{(\varepsilon,h,\tau)\to(0,0,0)}\Bbb E\left[K(X,\overline{X}_{\eps,h,\tau},\underline{X}_{\eps,h,\tau},X_{\eps,h,\tau})\right].
\end{equation}
Obviously, if $K$ is real bounded and \eqref{lsc_0} holds also for $-K$ then we get equalities and limits in \eqref{lsc_1} and \eqref{lsc_2}.
In particular, under the assumptions of Theorem \ref{main_thm_2},
$$
\|\overline{X}_{\eps,h,\tau}-X\|_{L^q((0,T);L^r(\mathcal O))}\to 0
$$
in probability as $(\varepsilon,h,\tau)\to(0,0,0)$ for every $r\in[1,\frac{d}{d-1})$ and every $q\in[1,\infty)$ such that $q(r-2)<r$.
\end{remark}

}
}

\section{A priori estimates}\label{sec_est}

The numerical approximation (\ref{num_tvf}) satisfies a discrete energy estimate.

\begin{lemma}\label{lem_energy}
Let $x^0, g\in\mathbb{L}^2$ and $T> 0$. Then there exists a constant $C>0$ depending only on {$T$ and on the constants in $(\mathbf B_1)$ and in Section \ref{sec_num}} such that the solutions of scheme \eqref{num_tvf} satisfy for any $\eps,h\in (0,1]$, $N\in \mathbb{N}$
\begin{align}\label{disc_supener}
\Bbb E\,\Big[\frac{1}{2}\sup_{i=1,\ldots,N}\nos{X^i}+\sum_{i=1}^N\Big(\frac{1}{4}\nos{X^i-X^{i-1}}&+\tau\mathcal J_\varepsilon(X^i)\Big)\Big]^2
\\
\nonumber&\leq{C}\left({\frac{1}{2}}+\frac{1}{2}\nos{x^0}+\be{\mathcal{O}}+\frac{\lambda}{2}\nos{g}\right)^2.
\end{align}
\end{lemma}
\begin{proof}
Analogically to \cite[Lemma~4.9]{num_stvf}, set $\vh=X^i$ in (\ref{num_tvf}) and obtain
\begin{align*}
\frac{1}{2}\nos{X^i}-\frac{1}{2}\nos{X^{i-1}} &+\frac{1}{4}\nos{X^i-X^{i-1}}+\tau\mathcal J_\varepsilon(X^i)
\\
\leq&\tau\mathcal J_\varepsilon(0)+\sum_{j=1}^N(B_j(X^{i-1}),X^{i-1})\xi^{i,j}_\tau+\left\|\sum_{j=1}^N B_j(X^{i-1})\xi^{i,j}_\tau\right\|^2\nonumber
\end{align*}
for $1\le i\le N$.
{If we define
\begin{align*}
a^i&=\frac 12+\frac{1}{2}\|X^i\|^2+\sum_{j=1}^i\left(\frac{1}{4}\|X^j-X^{j-1}\|^2+\tau\mathcal J_\varepsilon(X^j)\right)
\\
b&=\tau\mathcal J_\varepsilon(0),\quad c^i=\sum_{j=1}^N(B_j(X^{i-1}),X^{i-1})\xi^{i,j}_\tau,\quad d^i=\left\|\sum_{j=1}^N B_j(X^{i-1})\xi^{i,j}_\tau\right\|^2\,,
\end{align*}
then
$$
a^i-a^{i-1}\le b+c^i+d^i,\qquad i=1,\dots,N.
$$
If $a^{i-1}\in L^2(\Omega)$ then $X^{i-1}\in L^4(\Omega;\Bbb L^2)$, $c^i\in L^2(\Omega)$, $d^i\in L^2(\Omega)$ and so, 
by induction, $a^i\in L^2(\Omega)$ for every $i=0,\dots,N$. Next, observe that $c^i$ is a square integrable martingale difference and that
$$
\E{(c^i)^2}\le C_B \tau\E{(a^{i-1})^2},\qquad\E{(d^i)^2}\le C_B \tau^2\E{(a^{i-1})^2}
$$
where $C_B$ depends only on the growth constants in $(\mathbf B_1)$ and in the assumption $\E{(\xi^{i,j}_\tau)^2}\le C\tau^2$. Let us define
$$
a_*^i=\max_{j=0,\dots,i}a^j,\qquad i=0,\dots,N.
$$
Then
$$
a_*^i\le(a^0+Nb)+{\max_{j=1,\dots,i}}\left|\sum_{\ell=1}^j c^\ell\right|+\sum_{j=1}^id^j,\qquad i=1,\dots,N
$$
and
$$
(a_*^i)^2\le 3(a^0+Nb)^2+3{\max_{j=1,\dots,i}}\left|\sum_{\ell=1}^jc^\ell\right|^2+3N\sum_{j=1}^i(d^j)^2,\qquad i=1,\dots,N.
$$
Hence, by the discrete Burkholder-Davis-Gundy inequality, we obtain
\begin{align*}
\E{(a_*^i)^2}&\le 3(a^0+Nb)^2+3C_2\sum_{j=1}^i\E{(c^j)^2}+3N\sum_{j=1}^i\E{(d^j)^2}
\\
&\le 3(a^0+Nb)^2+3C_2\mathbf c\tau\sum_{j=1}^i\E{(a^{j-1})^2}+3\mathbf c\tau^2N\sum_{j=1}^i\E{(a^{j-1})^2}
\\
&\le 3(a^0+Nb)^2+\frac{K_{C,T}}{N}\sum_{j=1}^i\E{(a_*^{j-1})^2},\qquad i=1,\dots,N
\end{align*}
and we get the result by the discrete Gronwall lemma.

}
\end{proof}

Next, we estimate the discrete time increments of the numerical solution.
\begin{lemma}\label{lem_incr}
For any $0\le n\le \ell+n\le N$ it holds that
$$
\mathbb{E} [ \|X^{n+\ell}-X^{n}\|_{\Bbb H^{-1}}^4 ]\leq C t_\ell^2\,,
$$
{where $C$ does not depend on $\eps$, $h$, $\tau$.}
\end{lemma}
\begin{proof}
For any $v\in \mathbb{H}^1_0$ we set $v_h=P_h v$ in \eqref{num_tvf} and get after summing up for $i=n+1,\dots,n+\ell$ 
by the definition of projection $P_h$ that
\begin{align*}
& \ska{X^{n+\ell}-X^{n},v_h} =  \ska{X^{n+\ell}-X^n,P_h v} \leq  \tau \sum_{i=n+1}^{n+\ell}\left\|\frac{\nabla X^i}{\sqrt{|\nabla X^i|^2+\varepsilon^2}}\right\| \|\nabla P_hv\| \\
&\qquad +\tau\sum_{i=n+1}^{n+\ell}\lambda\big(\|X^i\| +\| {g}\|\big)\|P_hv\| + \Big\|\sum_{i=n+1}^{n+\ell}\sum_{j=1}^NB_j(X^{i-1})\xi^{i,j}_{\tau}\Big\|  \|P_h v\| &&\forall v \in \mathbb{H}^1_0. \nonumber
\end{align*}
{On noting that that $\left|\fe{\cdot}\right|\leq 1$
we deduce by the stability of the $\mathbb{L}^2$ projection $\|P_hv\|_{\Bbb H^1_0}\le\kappa\|v\|_{\Bbb H^1_0}$ that
$$
\|X^{n+\ell}-X^{n}\|_{\Bbb H^{-1}}\le Ct_{\ell}[1+\max_{i=1,\dots,N}\,\{\|X^i\|\})]+\left\|\sum_{i=n+1}^{n+\ell}\sum_{j=1}^NB_j(X^{i-1})\xi^{i,j}_\tau\right\|\,.
$$
Hence, we obtain
$$
\E{\|X^{n+\ell}-X^{n}\|^4_{\Bbb H^{-1}}}\le ct^4_{\ell}+ct^2_{\ell}\,,
$$
by the Burkholder-Rosenthal inequality, Lemma \ref{lem_energy} and linear growth of $B:\Bbb L^2\to\mathscr L_2(\ell_2;\Bbb L^2)$. 
Indeed, the martingale difference
$$
d_i=\sum_{j=1}^NB_j(X^{i-1})\xi^{i,j}_\tau\,,
$$
satisfies for $p\in\{2,4\}$
\begin{align*}
\sum_{i=n+1}^{n+l}\Bbb E\,\left[\|d_i\|^p|\mathcal F^{i-1}_\tau\right]&\le c_\kappa\tau^\frac p2\sum_{i=n+1}^{n+l}\left[\sum_{j=1}^N\|B_j(X^{i-1})\|^2\right]^\frac{p}{2}
\le c_\kappa\tau^\frac p2\sum_{i=n+1}^{n+l}\|B(X^{i-1})\|^p_{\mathscr L_2(\ell_2;\Bbb L^2)}
\\
&\le c_\kappa\tau^\frac p2\sum_{i=n+1}^{n+l}[1+\|X^{i-1}\|]^p
\le c_\kappa t^\frac p2_\ell[1+\max_{i=1,\dots,N}\,\{\|X^i\|^p\}].
\end{align*}
}

\end{proof}

\begin{lemma}\label{lem_3} Let $u^0\in\Bbb V_h$, let $G^1,\dots,G^N$ and $H^{0,j},\dots,H^{N-1,j}$ be $\mathcal F_\tau$-adapted random variables in $L^2(\Omega;\Bbb V_h)$ for every $j\in\{1,\dots,N\}$ and define
\begin{equation}\label{uidef}
U^i=u^0-\tau\sum_{\ell=1}^iG^\ell+\sum_{\ell=1}^i\sum_{j=1}^NH^{\ell-1,j}\xi^{\ell,j}_\tau,\qquad i\in\{0,\dots,N\}.
\end{equation}
Then
\begin{align*}
\frac{1}{2}\E{\|X^i-U^i\|^2}+\tau\sum_{\ell=1}^i\E{\mathcal J_\varepsilon(X^\ell)}&\le\frac{1}{2}\|x^0-u^0\|^2+\tau\sum_{\ell=1}^i\E{\mathcal J_\varepsilon(U^\ell)+(G^\ell,X^\ell-U^\ell)}
\\
&+\frac{\tau}{2}\sum_{\ell=1}^i\sum_{j=1}^N\E{\|P_h B_j(X^{\ell-1})-H^{\ell-1,j}\|^2},\quad 0\le i\le N.
\end{align*}
\end{lemma}
\begin{proof}
We denote $D^i=X^i-U^i$ and use (\ref{num_tvf}), (\ref{uidef}) to deduce the formula for $(D^i-D^{i-1},v_h)$ for $v_h\in\Bbb V_h$.
We then set $v_h=D^i$, use that $(B_j(X^{i-1}),D^{i})=(P_h B_j(X^{i-1}),D^{i})$ and proceed as in the proof of Lemma \ref{lem_energy}.
\end{proof}

\section{Compactness in locally convex path spaces}\label{SLCS}
{In this section, $Y$ stands for a Hausdorff locally convex space (typically a Hilbert space equipped with the strong or the weak topology), $Y^{[0,T]}$ denotes the space of functions from $[0,T]$ to $Y$ on which we consider the topology of uniform convergence $\tau_{\mathbf u}$. We also define the subspaces $Q_n([0,T];Y)$, $n\in\Bbb N$ spanned by the functions $f\in Y^{[0,T]}$ that are constant on every interval $(t_{i-1}^n,t^n_i)$ for $1\le i\le n$ where $t^k_j=jT/k$, the Hausdorff locally convex path spaces
$$
Q_\infty([0,T];Y)=\bigcup_{n=1}^\infty Q_n([0,T];Y),\qquad Q([0,T];Y)=\overline{Q_\infty([0,T];Y)}\,,
$$
and an important $F_\sigma$ subset of $Q([0,T];Y)$
$$
Q_c([0,T];Y)=Q_\infty([0,T];Y)\cup C([0,T];Y)\,,
$$
that contains both step-functions on equidistant partitions of $[0,T]$ and continuous functions, equipped with the uniform convergence topology, that is best suitable for our purposes in the sequel when piecewise constant processes will converge uniformly to a continuous process. The space of continuous $Y$-valued functions $C([0,T];Y)$ is also equipped with the topology of uniform convergence. 

Further, we define the space
$$
Q_{c,BV}([0,T];\Bbb L^2_w)=\left\{f\in Q_c([0,T];\Bbb L^2_w):\,\int_0^T\|f(s)\|_{BV(\mathcal O)}\,ds<\infty\right\}\,,
$$
as an $F_\sigma$ subset of $Q([0,T];\Bbb L^2_w)$.

Finally, if $M$ is a subset of $Q([0,T];Y)$, we define
$$
M^\uparrow_n=M\setminus\bigcup_{m=1}^{n-1}Q_m([0,T];Y).
$$

\begin{remark}\label{rem_1_ct} Every $f\in Q([0,T];Y)$, as a uniform limit of functions in $Q_\infty([0,T];Y)$, is bounded and also continuous at every {$x=Tr$} for some irrational number $r$. In particular, $f$ is continuous with an exception of an at most countable set and, as such, $f$ is Borel measurable. 
\end{remark}

\begin{remark}\label{rem_3_seq_comp} 
If $Y$ is sequentially complete then $Q([0,T];Y)$ coincides with the space of functions $f\in Y^{[0,T]}$ that are continuous at every $t\in(T\Bbb Q^\complement)\cap[0,T]$ and that have right and left limits at every $t\in[0,T]$.
\end{remark}

\begin{remark} The space $Q([0,T];Y)$ can be also equipped (alternatively) with the Skorokhod topology defined by neighbourhoods 
$$
N_{O,\varepsilon}(f)=\{g:\,\exists\mu\text{ such that }\gamma(\mu)<\varepsilon\text{ and }g(\mu(t))-f(t)\in O\text{ for every }t\in[0,T]\}
$$
where $O$ is an absolutely convex neighbourhood of zero in $Y$, $\varepsilon>0$, $\mu$ is an increasing  bi-Lipschitz continuous homeomorphisms of $[0,T]$ onto $[0,T]$ and $\gamma(\mu)=\|\log\mu^\prime\|_{L^\infty}$. But the Skorokhod topology is strictly weaker than the topology of uniform convergence. In other words, convergence in $Q([0,T];Y)$ implies convergence in the Skorokhod topology but not vice versa. Thus, for our purposes, the space $Q([0,T];Y)$ with the topology of uniorm convergence is the better choice.
\end{remark}

In the next theorem, we characterize compact sets in $Q_c([0,T];Y)$ which play an essential role in this paper. To this end, we present an Arzela-Ascoli theorem.

\begin{thm}\label{comp_char} Let $M$ be a non-empty subset in $Q([0,T];Y)$ and consider the following:
\begin{itemize}
\item[(i)] $\left\{f(t):f\in M\right\}$ is relatively compact in $Y$ for every $t\in [0,T]$;
\item[(ii)] for every $O$ being a neighbourhood of zero in $Y$, there exist $m\in\Bbb N$ and $\delta>0$ such that
$$
\forall|t-s|\le\delta\quad\text{ and }\quad\forall f\in M_m^\uparrow\quad\text{ one has }\quad f(t)-f(s)\in O;
$$
\item[(iii)] the closure of $M$ in $(Y^{[0,T]},\tau_{\mathbf u})$ is a compact subset of $Q_c([0,T];Y)$;
\item[(iv)] $\left\{f(t):\,t\in[0,T],\,f\in M\right\}$ is relatively compact in $Y$.
\end{itemize}
Then
$$
[\,(i)\quad \&\quad (ii)\,]\quad\Longleftrightarrow\quad (iii)\quad\Longrightarrow\quad(iv).
$$
\end{thm}

\begin{proof}
See Section \ref{proof_comp_char}.
\end{proof}

\begin{remark}\label{rem_4} If $Y$ is sequentially complete and $M$ is relatively compact in $Q([0,T];Y)$ then (iv) in Theorem \ref{comp_char} still holds with the same proof.
\end{remark}

\begin{corollary}\label{cor_1_me}
If compacts of $Y$ are metrizable and $M$ satisfies (iv) in Theorem \ref{comp_char} then $M$ is also metrizable. 
\end{corollary}

\begin{proof}
See Section \ref{proof_cor_1_me}.
\end{proof}

Now we provide an easy test for checking Borel measurability of $Q_c([0,T];Y)$-valued random variables. It turns out that pointwise measurability and Borel measurability coincide for mappings with a $\sigma$-compact range in $Q_c([0,T];Y)$ provided that compact sets in $Y$ are metrizable.

\begin{corollary}\label{cor_2_me} Let compacts of $Y$ be metrizable and let $M$ be $\sigma$-compact in $Q_c([0,T];Y)$. Then
$$
V\in\mathscr B(Q([0,T];Y))\quad\Longleftrightarrow\quad V\in\mathcal Y_T\,,
$$
holds for every $V\subseteq M$ where 
$$
\mathcal Y_T=\sigma(\pi_s:\,s\in(T\Bbb Q)\cap[0,T]),\qquad\pi_s:Q([0,T];Y)\to Y:f\mapsto f(s).
$$
\end{corollary}

\begin{proof}
See Section \ref{proof_cor_2_me}.
\end{proof}

\begin{remark}\label{rem_qps} Let us recall that a compact $K$ is metrizable if and only if there exists a countable family of real continuous functions on $K$ separating points of $K$ (see e.g. \cite{jakub}). In case of Hausdorff locally convex spaces $Y$, those functions can be chosen in such a way that they are linear and continuous on $Y$. Hence compacts are metrizable in all spaces where there exists a countable family of continuous functions separating points of that space. In particular, compact sets are metrizable e.g. in analytic spaces (see e.g. \cite[Corollary 6.7.8]{bog}) among which all separable Fr\'echet spaces equipped with any locally convex topology weaker than or equal to the metric one belong.
\end{remark}

\begin{example}\label{pc_comp} If $K$ is a set in $Y$ then we denote by $C_n([0,T];K)$ the space of functions $f:[0,T]\to K$ that satisfy
$$
f(t)=\frac{t-t_{i-1}}{\tau}f(t_i)+\frac{t_i-t}{\tau}f(t_{i-1}),\qquad t\in[t_{i-1},t_i]\,,
$$
for every $i\in\{1,\dots,n\}$ where $t_i=i\tau$ and $\tau=T/n$. If $K$ is compact then $C_n([0,T];K)$ is compact in $C([0,T];Y)$.
\end{example}

\begin{proof}
Indeed, $C_n([0,T];K)$ is closed. Now, if $O$ is an absolutely convex neighbourhood of zero then $K\subseteq\lambda O$ for some $\lambda>0$, and so $f(t)-f(s)\in 2\lambda\tau^{-1}(t-s)O$ holds for every $s,t\in[0,T]$ and every $f\in C_n([0,T];K)$. Hence $C_n([0,T];K)$ is relatively compact by Theorem \ref{comp_char}.
\end{proof}

We will need the following version of the Prokhorov theorem.

\begin{thm}\label{prokhorov} Let $Z$ be a completely regular topological space, let $\{\mu_n\}$ be Borel probability measures such that there exist metrizable compacts $K_j$ such that
$$
\sup_j\,[\inf_n\,\mu_n(K_j)]=1.
$$
Then there exists a subsequence $\{\mu_{n_k}\}$ that converges to a Radon probability measure $\mu$ on $Z$.
\end{thm}

\begin{proof}
See \cite[Theorem 8.6.7.]{bog}.
\end{proof}

Weak convergence of tight probability measures is actually more powerful than it might seem. Let us present a reinforcement of the Portmanteau theorem, cf. \cite[Lemma 1.10]{opw}.

\begin{proposition}\label{portmanteau} Let $Z$ be a completely regular topological space, let $\{\mu_n\}$ and $\mu$ be Radon probability measures on $Z$ such that $\langle f,\mu_n\rangle\to\langle f,\mu\rangle$ for every $f\in C_b(Z)$ and, for every $r>0$, there exist metrizable closed sets $K_{r,n}\searrow K_{r,\infty}$ such that
$$
\mu_n(K_{r,n})\ge 1-r,\qquad n\in\Bbb N.
$$
Let $F_n,F:Z\to[-\infty,\infty]$ be such that $F_n|_{{K_{r,n}}}$, $F|_{{K_{r,n}}}$ are $\mathscr B({K_{r,n}})$-measurable for every {$r>0$ and} $n\in\Bbb N$, and denote by $\mu^*$ the outer measure associated with $\mu$. Then $F_n$ is $\mu_n$-measurable for every $n\in\Bbb N$, $F$ is $\mu$-measurable and the following holds:
\begin{enumerate}
\item If $F_n$ and $F$ are non-negative and $\mu^*(D_r)=0$ for every $r\in(0,1)$ where
$$
D_r=\{x\in K_{r,\infty}:\,\exists x_n\in K_{r,n},\,x_n\to x,\,\liminf F_n(x_n)<F(x)\}\,,
$$
then
$$
\int_ZF\,\dd\mu\le\liminf\int_ZF_n\,\dd\mu_n.
$$
\item If $\mu^*(D_r)=0$ for every $r\in(0,1)$ where
$$
D_r=\{x\in K_{r,\infty}:\,\exists x_n\in K_{r,n},\,x_n\to x,\,\limsup|F_n(x_n)-F(x)|>0\}\,,
$$
and
$$
\lim_{R\to\infty}\left[\sup_{n\in\Bbb N}\int_{[|F_n|>R]}|F_n|\,\dd\mu_n\right]=0\,,
$$
then
$$
\lim\int_ZF_n\,\dd\mu_n=\int_ZF\,\dd\mu.
$$
\end{enumerate}
\end{proposition}

\begin{proof}
See Section \ref{proof_portmanteau}.
\end{proof}

}

\section{Tightness properties of the numerical approximation}\label{sec_conv}

We consider the interpolants $X_\tau$, $\overline{X}_\tau$, $\underline{X}_\tau$
defined in (\ref{xtaulin}), (\ref{xtaubar}), respectively.
As in the previous section, to simplify the notation, the dependence of $X_\tau$, $\overline{X}_\tau$ and $\underline{X}_\tau$ on $\varepsilon$, $h$ and $\tau$ will not be displayed for clarity reasons until it matters.

The next lemma is a direct consequence of the a priori estimates in Lemma~\ref{lem_energy}.
\begin{lemma}\label{un_est}
The interpolants of the numerical solution of the scheme (\ref{num_tvf}) satisfy the following bounds:
\begin{align}
& \mathbb{E}\Big[ \|\overline{X}_\tau\|^2_{L^1(0,T;\mathbb{W}_0^{1,1})}\Big]  \leq C, \,\,{\mathbb{E}\Big[ \|\underline{X}_\tau\|^2_{L^1(\tau,T;\mathbb{W}_0^{1,1})}\Big]  \leq C\,,\mathbb{E}\Big[ \|X_\tau\|^2_{L^1(\tau,T;\mathbb{W}_0^{1,1})}\Big]  \leq C\,,}
\\
& \mathbb{E}\Big[\|\overline{X}_\tau\|_{L^\infty(0,T;\mathbb{L}^2)}^4\Big]  \leq C, \,\, \mathbb{E}\Big[\|\underline{X}_\tau\|_{L^\infty(0,T;\mathbb{L}^2)}^4\Big] \leq C\,,
\,\, \mathbb{E}\Big[\|X_\tau\|_{C([0,T];\mathbb{L}^2)}^4\Big] \leq C\,,
\\
& 
{\mathbb{E}\Big[\|\overline{X}_\tau-X_\tau\|_{L^q(0,T;\mathbb{L}^2)}^4\Big]  \leq C\tau^\frac{4}{q}, \,\, \mathbb{E}\Big[\|\overline{X}_\tau-\underline{X}_\tau\|_{L^q(0,T;\mathbb{L}^2)}^4\Big] \leq C\tau^\frac{4}{q}\,,\label{un_est_3}
}
\end{align}
{where $C$ does not depend on $\eps$, $h$, $\tau$ and $q\in[2,\infty]$.}
\end{lemma}

Furthermore, by Lemma~\ref{lem_incr} the following time-fractional bounds hold for the piecewise linear interpolant. 

\begin{lemma}\label{mod_cont_m1}
Let $\mathbf m$ denote the modulus of continuity of $\Bbb H^{-1}$-valued functions on $[0,T]$
$$
\mathbf m(f,\delta):=\sup\,\{\|f(t)-f(s)\|_{\Bbb H^{-1}}:\,s,t\in[0,T],\,|t-s|\le\delta\}.
$$
Then the following estimate holds for $\alpha\in (0,\frac{1}{2})$ and $s\in (0,\frac{1}{4})$
\begin{align*}
\E{\|X_\tau\|_{W^{\alpha, 4}(0,T;\Bbb H^{-1})}^4}&\leq C,&\E{\sup_{\delta>0}\,\{\delta^{-s}\mathbf{m}(X_\tau,\delta)\}}&\leq C,&\delta>0\,,
\\
\E{\sup_{\delta>0}\,\{(\delta+\tau)^{-s}\mathbf{m}(\overline{X}_\tau,\delta)\}}&\leq C,&\E{\sup_{\delta>0}\,\{(\delta+\tau)^{-s}\mathbf{m}(\underline{X}_\tau,\delta)\}}&\leq C,&\delta>0\,,
\end{align*}
where $C$ does not depend on $\eps$, $h$, $\tau$.
\end{lemma}
\begin{proof}
Use Lemma \ref{lem_incr}, Lemma \ref{besov_discr} and the inequality
\begin{equation*}
\max\,\{{\mathbf m}(\overline{X}_\tau,\delta),{\mathbf m}(\underline{X}_\tau,\delta)\}\le{\mathbf m}(X,\delta+\tau),\qquad \delta>0.
\end{equation*}
\end{proof}

{With the notation and the parameters from Lemma \ref{mod_cont_m1}, for $R>0$ and $a\in[0,T]$, writing shortly $Q_n$ for $Q_n([0,T];\Bbb L^2_w)$ and $C$ for $C([0,T];\Bbb L^2_w)$, we consider the sets
\begin{align*}
V_{R,n,a}&=\{f\in Q_n:\,&\sup_{t\in[0,T]}\|f(t)\|\le R,\,&\sup_{\delta>0}\frac{\mathbf m(f,\delta)}{(\delta+T/n)^s}\le R,\,&\int_a^T\|f(s)\|_{BV(\mathcal O)}\,\dd s\le R\}\,,
\\
V_{R,\infty,a}&=\{f\in C:&\sup_{t\in[0,T]}\|f(t)\|\le R,&\sup_{\delta>0}\frac{\mathbf m(f,\delta)}{\delta^s}\le R,&\int_a^T\|f(s)\|_{BV(\mathcal O)}\,\dd s\le R\}\,,
\end{align*}
$$
V^m_{R,b}=V_{R,\infty,b_*}\cup\bigcup_{n\in[m,\infty]}V_{R,n,b_n},\qquad b_*:=\limsup_{n\to\infty} b_n.
$$

\begin{proposition}\label{tight_part} The random variables $\overline{X}_\tau$, $\underline{X}_\tau$, $X_\tau$ are Borel measurable as mappings from $(\Omega_\tau,\mathscr F_\tau,\Bbb P_\tau)$ to $Q_c([0,T];\Bbb L^2_w)$, for every $m,n\in\Bbb N$ and $a\in[0,T]^{\Bbb N\cup\{\infty\}}$, the sets $V^m_{R,a}$ and $V_{R,n,a_n}$ are compact in $Q_c([0,T];\Bbb L^2_w)$, the sets $V_{R,\infty,a_\infty}$ are compact in $C([0,T];\Bbb L^2_w)$, and
$$
\Bbb P_\tau\,[\overline{X}_\tau\notin V_{R,N,0}]\le\frac{C}{R},\qquad\Bbb P_\tau\,[\underline{X}_\tau\notin V_{R,N,T/N}]\le\frac{C}{R},\qquad\Bbb P_\tau\,[X_\tau \notin V_{R,\infty,T/N}]\le\frac{C}{R}\,,
$$
holds for every $R>0$ where $C$ does not depend on $\eps$, $h$, $\tau$ and $R$. In particular, the laws
$$
\Bbb P_\tau\,[\overline{X}_{\eps,h,\tau}\in\cdot\,],\qquad\Bbb P_\tau\,[\underline{X}_{\eps,h,\tau}\in\cdot\,]\qquad\Bbb P_\tau\,[X_{\eps,h,\tau}\in\cdot\,]\,,
$$
are tight on $Q_{c,BV}([0,T];\Bbb L^2_w)$, $Q_c([0,T];\Bbb L^2_w)$ and $C([0,T];\Bbb L^2_w)$ resp.  with respect to $\eps$, $h$, $\tau$.
\end{proposition}

\begin{proof} $\overline{X}_\tau$, $\underline{X}_\tau$ and $X_\tau$ are clearly $\mathcal Y_T$-measurable and $Q_N([0,T];\Bbb L^2_w)$ and $C_N([0,T];\Bbb L^2_w)$ are $\sigma$-compact in $C([0,T];\Bbb L^2_w)$ by Theorem \ref{comp_char} and Example \ref{pc_comp}. Hence $\overline{X}_\tau$, $\underline{X}_\tau$ and $X_\tau$ are Borel measurable by Corollary \ref{cor_2_me} as compact sets in $Q_c([0,T];\Bbb L^2_w)$ and $C([0,T];\Bbb L^2_w)$ are metrizable (Remark \ref{rem_qps}). Now the sets $V_{R,n,a_n}$ and $V_{R,a}$ are closed and relatively compact in $Q_c([0,T];\Bbb L^2_w)$ and the sets $V_{R,\infty,a_\infty}$ are closed and relatively compact in $C([0,T];\Bbb L^2_w)$ by Theorem \ref{comp_char}, as the weak topology and the $\Bbb H^{-1}$-topology coincide on bounded sets in $\Bbb L^2$. In the proof of closedness of the above sets, we use the fact that there exists a countable set $\mathcal H$ of smooth compactly supported functions such that
$$
\|g\|_{BV(\mathcal O)}=\sup\,\{(g,\varphi):\,\varphi\in\mathcal H\},\quad \text{ for } g\in L^1_{loc}(\mathcal O)\,,
$$
holds, e.g., by \cite[Proposition 3.6]{AFP} and by separability of $C^\infty_c(\mathcal O)$. Hence
$$
f\mapsto\int_0^T\|f(s)\|_{BV(\mathcal O)}\,\dd s,
$$
as a supremum of continuous functions, is lower semicontinuous on $Q([0,T];\Bbb L^2_w)$.

The tightness then follows directly from Lemma \ref{un_est} and Lemma \ref{mod_cont_m1}.
\end{proof}

In the next lemma we obtain the convergence of the noise variables to a Wiener process.

\begin{lemma}\label{donsker} Let $W^j_\tau$, $1\le j\le N$ be the piecewise linear processes on $[0,T]$ defined by
$$
W^j_\tau(t_i)=\sum_{\ell=1}^i\xi^{\ell,j}_\tau,\qquad 0\le i\le N\,,
$$
and $W^j_\tau$ is linear on $[t_{i-1},t_i]$ for every $0<i\le N$ where $\tau=T/N$ and $t_i=i\tau$. We also define $W^j_\tau=0$ for $j>N$. Then the laws of $W^j_\tau$ converge to the Wiener measure on $C\,[0,T]$ as $\tau\to 0$, for every $j\in\Bbb N$.
\end{lemma}

\begin{proof}
Let $s\in(1/4,1/2)$. Then, 
$$
\E{\left|W^j_\tau(t_n)-W^j_\tau(t_{n-\ell})\right|^4}\le C_\kappa t_{\ell}^2,\qquad 1\le\ell\le n\le N\,,
$$
hence, by Lemma \ref{besov_discr} we get
\begin{equation}\label{donsker_besov}
\E{\,\|W^j_\tau\|^4_{B^s_{4,4}(0,T)}}\le C_{\kappa,s,T}.
\end{equation}
In particular, since $B^s_{4,4}(0,T)$ is embedded compactly in $C^\alpha([0,T])$ for every $0<\alpha<s-\frac{1}{4}$ e.g. by \cite[Corrolary 26]{simon}, the laws of 
$\{W^j_\tau\}$ are tight on $\mathscr B(C([0,T]))$. 
Since $(W^j_\tau(s_0),\dots,W^j_\tau(s_k))$ converge in law to the law of $(W(s_0),\dots,W(s_k))$ where $W$ is a Wiener process, e.g. by Theorem 18.2 in \cite{billingsley}, we get the claim.
\end{proof}
}

{
Let us consider the completely regular space with metrizable compacts (see Remark \ref{rem_qps})
$$
\mathbf Z=Q_{c,BV}([0,T];\Bbb L^2_w)\times Q_c([0,T];\Bbb L^2_w)\times C([0,T];\Bbb L^2_w)\times C([0,T])\times C([0,T])\times C([0,T])\times\dots,
$$
define the projections
\begin{align}\label{proj}
\nonumber
&S^1:\mathbf Z\to Q_{c,BV}([0,T];\Bbb L^2_w)&(f^1,f^2,f^3,w^1,w^2,w^3,\dots)\mapsto f^1,
\\ \nonumber
&S^2:\mathbf Z\to Q_c([0,T];\Bbb L^2_w)&(f^1,f^2,f^3,w^1,w^2,w^3,\dots)\mapsto f^2,
\\ 
&S^3:\mathbf Z\to C([0,T];\Bbb L^2_w)&(f^1,f^2,f^3,w^1,w^2,w^3,\dots)\mapsto f^3,
\\ \nonumber
&W^j:\mathbf Z\to C([0,T])&(f^1,f^2,f^3,w^1,w^2,w^3,\dots)\mapsto w^j,
\end{align}
and the canonical filtration on $\mathbf Z$
$$
\mathcal Z_t=\sigma(S^1_s,S^2_s,S^3_s,W^j_s:\,s\in[0,t],\,j\in\Bbb N),\qquad t\in[0,T].
$$
If $\nu$ is a probability measure on $\mathscr B(\mathbf Z)$ then $\mathcal Z^\nu_t$ stands for the augmentation of $\mathcal Z_t$ by $\nu$-negligible Borel sets.

\begin{corollary}\label{main_cor} The random variables 
$$
Z_{\eps,h,\tau}=(\overline{X}_{\eps,h,\tau},\underline{X}_{\eps,h,\tau},X_{\eps,h,\tau},W^1_\tau,W^2_\tau,W^3_\tau,\dots)
$$
are Borel measurable as mappings from $(\Omega_\tau,\mathscr F_\tau,\Bbb P_\tau)$ to $\mathbf Z$ and their laws under $\Bbb P_\tau$ are tight on $\mathscr B(\mathbf Z)$ with respect to $\eps$, $h$, $\tau$. In particular, every sequence $(\varepsilon_n,h_n,\tau_n)$ has a subsequence $(\varepsilon_{n_k},h_{n_k},\tau_{n_k})$ such that laws of $Z_{\eps_{n_k},h_{n_k},\tau_{n_k}}$ under $\Bbb P_{\tau_{n_k}}$ converge to a Radon probability measure $\mu$ on $\mathscr B(\mathbf Z)$.
\end{corollary}
\begin{proof} 
Since $\overline{X}_{\eps,h,\tau}$, $\underline{X}_{\eps,h,\tau}$ and $X_{\eps,h,\tau}$ take values in $\sigma$-compact subsets of $Q_{c,BV}([0,T];\Bbb L^2_w)$, $Q_c([0,T];\Bbb L^2_w)$ and $C([0,T];\Bbb L^2_w)$ respectively by Theorem \ref{comp_char} and Example \ref{pc_comp} and the fact that compact sets in all these spaces are metrizable (Remark \ref{rem_qps}), we get that $Z_{\eps,h,\tau}$ is Borel measurable e.g. by \cite[Lemma 6.4.2/ii]{bog}. Tightness follows from Proposition \ref{tight_part} and Lemma \ref{donsker} and convergence of a subsequence by Theorem \ref{prokhorov}.
\end{proof}

\section{Construction of a probabilistically weak SVI solution}\label{CoMaSo}}

Thanks to Corollary~\ref{main_cor}, in the sequel, we choose a subsequence
 $(\varepsilon_k,h_k,\tau_k)\to (0,0,0)$ such that the Borel laws of $Z_k=Z_{\eps_k,h_k,\tau_k}$ under $\Bbb P_{\tau_k}$ converge to a Radon probability measure $\mu$ on $\mathscr B(\mathbf Z)$.

\begin{lemma}\label{cont_meas} Let $F_k,F:\mathbf Z\to[-\infty,\infty]$ be such that $F_k|_K$ and $F|_K$ are $\mathscr B(K)$-measurable for every compact $K$ in $\mathbf Z$ (e.g., sequentially lower semicontinuous) and every $k\in\Bbb N$). 
Further, assume that one of the following
\begin{itemize}
\item[(a)] $F_k$ and $F$ are non-negative and 
$$
F(f,g,h,w^1,w^2,\dots)\le\liminf_{k\to\infty}F_k(f_k,g_k,h_k,w^1_k,w^2_k,\dots),
$$
\item[(b)] $\lim_{k\to\infty}F_k(f_k,g_k,h_k,w^1_k,w^2_k,\dots)=F(f,g,h,w^1,w^2,\dots)$ and
\begin{equation}\label{unif_int_fk}
\lim_{R\to\infty}\left[\sup_{k\in\Bbb N}\Bbb E_{\tau_k}\left[\mathbf 1_{[|F_k(Z_k)|>R]}|F_k(Z_k)|\right]\right]=0
\end{equation}
\end{itemize}
holds for every
\begin{itemize}
\item[(i)] $f_k\to f$ in $Q_{c,BV}([0,T];\Bbb L^2_w)$, $\sup_k\int_0^T\|f_k(s)\|_{BV(\mathcal O)}\,ds<\infty$, $f\in C([0,T];\Bbb L^2_w)$,
\item[(ii)] $g_k\to g$ in $Q_c([0,T];\Bbb L^2_w)$, $\sup_k\int_{\tau^*_k}^T\|g_k(s)\|_{BV(\mathcal O)}\,ds<\infty$, $g\in C([0,T];\Bbb L^2_w)\cap Q_{c,BV}([0,T];\Bbb L^2_w)$,
\item[(iii)] $h_k\to h$ in $C([0,T];\Bbb L^2_w)$, $\sup_k\int_{\tau^*_k}^T\|h_k(s)\|_{BV(\mathcal O)}\,ds<\infty$, $h\in Q_{c,BV}([0,T];\Bbb L^2_w)$,
\item[(iv)] $w^j_k\to w^j$ in $C([0,T])$ for every $j\in\Bbb N$
\end{itemize}
where $\tau^*_k=\max\{\tau_i:\,i\ge k\}$. If (a) holds then
$$
\int_{\mathbf Z}F\,d\mu\le\liminf_{k\to\infty}\Bbb E_{\tau_k}\left[F_k(Z_k)\right].
$$
If (b) holds then
$$
\int_{\mathbf Z}F\,d\mu=\lim_{k\to\infty}\Bbb E_{\tau_k}\left[F_k(Z_k)\right].
$$
\end{lemma}

\begin{proof}
The sets
$$
\mathcal K_{R,n}=\left[\bigcup_{m\in[n,\infty]}V_{R,m,0}\right]\times\left[\bigcup_{m\in[n,\infty]}V_{R,m,T/m}\right]\times V_{R,\infty,T/n}\times C([0,T])\times C([0,T])\times\dots\,,
$$
are closed, metrizable and decreasing in the second variable, 
$$
\mathcal K_{R,\infty}:=\bigcap_{n=1}^\infty\mathcal K_{R,n}=V_{R,\infty,0}\times V_{R,\infty,0}\times V_{R,\infty,0}\times C([0,T])\times C([0,T])\times\dots\,,
$$
and
$$
\Bbb P_{\tau_k}\,[Z_k\notin\mathcal K_{R,T/\tau^*_k}]\le\Bbb P_{\tau_k}\,[Z_k\notin\mathcal K_{R,T/\tau_k}]\le\frac{C}{R}\,,
$$
by Proposition \ref{tight_part}. The rest follows from Proposition \ref{portmanteau}.
\end{proof}

\begin{remark} {From the definition of the topological space $Q_{c,BV}([0,T];\Bbb L^2_w)$} we observe  that $\overline{X}_{\eps_k,h_k,\tau_k}$ converges in a significantly stronger (hence better) sense than $\underline{X}_{\eps_k,h_k,\tau_k}$ and $X_{\eps_k,h_k,\tau_k}$.
\end{remark}

\begin{corollary}\label{first_prop} If $\alpha\in(0,\frac{1}{2})$ then the following holds:
\begin{itemize}
\item[(I)] The $\Bbb L^2$-valued processes $S^1$, $S^2$, $S^3$ and the real-valued processes $\Wproc$ are $(\mathcal Z_t)$-progressively measurable.
\item[(II)] $\mu(S^1=S^2=S^3)=1$.
\item[(III)] We have
$$
\int_{\mathbf Z}\left[\sup_{t\in[0,T]}\|S^3(t)\|^4+\|S^3\|_{W^{\alpha,4}(0,T;\Bbb H^{-1})}^4+\left(\int_0^T\|S^1(t)\|_{BV(\mathcal O)}{\,\dd t}\right)^2\right]\,\dd\mu<\infty\,.
$$
\item[(IV)] The $\sigma$-algebras $\mathcal Z_t$ and $\sigma(W^j(b)-W^j(a):\,t\le a\le b\le T,\,j\in\Bbb N)$ are $\mu$-independent.
\item[(V)] The processes $W^1,W^2,W^3,\dots$ are $\mu$-independent {$(\mathcal Z_t)$-}Brownian motions.
\end{itemize}
\end{corollary}

\begin{proof} (I) {follows from Remark \ref{rem_1_ct} as the processes $S^1$, $S^2$, $S^3$ are continuous with an exception of an at most countable set and they are $(\mathcal Z_t)$-adapted by definition, cf. \cite[Proposition 1.13]{Karatzas_Shreve_1988},} and (II), (III) from Lemma \ref{un_est}, Lemma \ref{mod_cont_m1} and Lemma \ref{cont_meas}.

As for (IV), it suffices to realize that
$$
\mathcal Z_t=\sigma((\varphi,S^1_s),(\varphi,S^2_s),(\varphi,S^3_s),W^j_s:\,s\in[0,t],\,j\in\Bbb N,\,\varphi\in\Bbb L^2).
$$
If $u\ge t+\tau_k$ then $\sigma(W^j(b)-W^j(a):\,u\le a\le b\le T,\,j\in\Bbb N)$ and $\mathcal Z_t$ are $\Bbb P_{\tau_k}(Z_k\in\cdot\,)$-independent, hence also $\mu$-independent by Lemma \ref{cont_meas}. Consequently, $\sigma(W^j(b)-W^j(a):\,t<a\le b\le T,\,j\in\Bbb N)$ and $\mathcal Z_t$ are $\mu$-independent but the former coincides with $\sigma(W^j(b)-W^j(a):\,t\le a\le b\le T,\,j\in\Bbb N)$ since the processes $W^j$ are continuous.

As for (V), the $\sigma$-algebras $\sigma(W^1),\sigma(W^2),\sigma(W^3),\dots$ are $\Bbb P_{\tau_k}(Z_k\in\cdot\,)$-independent, hence also $\mu$-independent by Lemma \ref{cont_meas}. And Lemma  \ref{donsker} yields that they are Brownian.
\end{proof}
}

\begin{thm}\label{main_full_1} 
{The process $S^3$ defined in (\ref{proj}) is an SVI solution on $(\mathbf Z,\mathscr B(\mathbf Z),(\mathcal Z^\mu_t),\mu)$ with Wiener processes $\Wproc$ (also defined in defined in (\ref{proj})) 
in the sense of Definition \ref{def_svi}.}
\end{thm}

\begin{proof}
The proof is divided into several steps. Recall that we consider the sub-sequence $(\varepsilon_k,h_k,\tau_k)\to (0,0,0)$ for $k\rightarrow 0$.

\medskip

{\bf(i)} First, we show that a discrete version {(\ref{disc5})} of \eqref{SVI_def} holds for simple step-processes $G$ and $H$. For let $0=s_0<\dots<s_m=T$, define $\Bbb R^{4M^2}$-valued 
 continuous mappings on $\mathbf Z$ as
$$
V^\alpha=((\varphi_\beta, S^1_{r^{\alpha}_\gamma}),(\varphi_\beta,S^2_{r^{\alpha}_\gamma}),(\varphi_\beta,S^3_{r^{\alpha}_\gamma}),W^j_{r^{\alpha}_\gamma}:\beta,\gamma,j\in\{1,\dots,M\}),\quad 0\le\alpha\le m\,,
$$
for some {$r^\alpha_\gamma\in[0,s_\alpha]$} and {$\varphi_\beta\in\Bbb L^2$} {where we consider the product with $\varphi_\beta$ to work with real-valued random variables,} and let
$$
g_\alpha,h_{\alpha,j}:\Bbb R^{4M^2}\to\Bbb H^1_0,\qquad\alpha\in\{0,\dots,m\},\,j\in\Bbb N\,,
$$
be $\Bbb H^1_0$-bounded continuous functions such that $h_{\alpha,j}=0$ for $j\ge j_0$ {and some arbitrary $j_0\in\Bbb N$, to simplify the argument.} 
We define
\begin{equation*}
G(t)=\sum_{\alpha=0}^{m-1}\mathbf 1_{(s_\alpha,s_{\alpha+1}]}(t)g_\alpha(V^\alpha),\qquad H_j(t)=\sum_{\alpha=0}^{m-1}\mathbf 1_{(s_\alpha,s_{\alpha+1}]}(t)h_{\alpha,j}(V^\alpha)\,,
\end{equation*}
and
\begin{equation}\label{def_I}
I(t)=u^0-\int_0^tG(s)\,\dd s+\sum_{j=1}^{j_0}\int_0^tH_j(s)\,\dd W^j.
\end{equation}
Setting $N_k=T/\tau_k$, $t_i:=i\tau_k$ for $i\in\{0,\dots,N_k\}$ then $G_{t_i}(Z_k)$ and $H_{j,t_i}(Z_k)$ are $\mathcal F^i_{\tau_k}$-measurable, Lemma~\ref{lem_3} yields
\begin{align*}
\frac{1}{2}\Bbb E_{\tau_k}[\|S^1_{t_i}(Z_k)&-P_{h_k}U^i(Z_k)\|^2]+\Bbb E_{\tau_k}\left[\int_0^{t_i}\mathcal J_{\varepsilon_k}(S^1_s(Z_k))\,\dd s\right]\le\frac{1}{2}\|x^0-u^0\|^2
\\
&+\sum_{\ell=1}^i\Bbb E_{\tau_k}\left[\int_{t_{\ell-1}}^{t_\ell}[\mathcal J_{\varepsilon_k}(P_{h_k}U^\ell(Z_k))+(P_{h_k}G_{t_\ell}(Z_k),S^1_s(Z_k)-U^\ell(Z_k))]\,\dd s\right]
\\
&+\frac{1}{2}\sum_{\ell=2}^i\Bbb E_{\tau_k}\left[\int_{t_{\ell-2}}^{t_{\ell-1}}\|P_{h_k} B(S^1_s(Z_k))-{P_{h_k}}H_{t_{\ell-1}}(Z_k)\|_{\mathscr L_2(\ell_2,\Bbb L^2)}^2\,\dd s\right]
\\
& +\frac{\tau_k}{2}\|P_hB(x^0)\|_{\mathscr L_2(\ell_2,\Bbb L^2)}^2\,,
\end{align*}
for $0\le i\le N_k$ where
\begin{equation}\label{def_U}
U^i=u^0-\tau_k\sum_{\ell=1}^iG(t_\ell)+\sum_{\ell=1}^i\sum_{j=1}^{N_k}(W^j(t_\ell)-W^j(t_{\ell-1}))H_j(t_{\ell-1})\qquad i\in\{0,\dots,N_k\}\,,
\end{equation}
{as $S^1_{t_i}(Z_k)=\overline{X}^i_{\varepsilon_k,h_k,\tau_k}$ by the definition of $S^1$ and $Z_k$.} For $N_k\ge j_0$ we deduce that
\begin{equation}\label{est_I_U_1}
\max_{1\le\ell\le N_k}\sup_{t\in[t_{\ell-1},t_\ell]}\|I(t)-U^\ell\|_{\Bbb H^1_0}\le C_G\tau_k+C_H\sum_{j=1}^{j_0}\mathbf {m}(W^j,\tau_k)\,,
\end{equation}
where $\mathbf {m}$ is the modulus of continuity of real-valued functions.

{In the following, we replace $U$ by $I$ in the {last but one} inequality above, we proceed term by term. We note that
\begin{equation}\label{est_I_U_2}
\Bbb E_{\tau_k}\left[\max_{1\le\ell\le N_k}{\|U^\ell(Z_k)\|_{\Bbb H^1_0}^2}+\sup_{s\in[0,T]}{\|S^1_s(Z_k)\|^2}\right]\le C\,,
\end{equation}
and 
\begin{equation}\label{est_modu}
\Bbb E_{\tau_k}\left[\mathbf m(W^j{(Z_k)},\tau_k)\right]^2\le C\tau_k^{2\theta}\,,
\end{equation}
hold for some $\theta\in(0,\frac{1}{4})$ by \eqref{donsker_besov}, the Doob maximal inequality for submartingales and Lemma \ref{un_est}. 
Next, we observe that
\begin{align*}
|\Bbb E_{\tau_k}[\|S^1_{t_i}(Z_k)&-P_{h_k}U^i(Z_k)\|^2]-\Bbb E_{\tau_k}[\|S^1_{t_i}(Z_k)-P_{h_k}I_{t_i}(Z_k)\|^2]|
\\
&\le\Bbb E_{\tau_k}[\|P_{h_k}U^i(Z_k)-P_{h_k}I_{t_i}(Z_k)\|^2]+4\sqrt{C}\left\{\Bbb E_{\tau_k}[\|P_{h_k}U^i(Z_k)-P_{h_k}I_{t_i}(Z_k)\|^2]\right\}^\frac12
\\
&\le C\tau_k^\theta\,,
\end{align*}
and
\begin{align*}
\Bbb E_{\tau_k}&\left[\int_{t_{\ell-1}}^{t_\ell}|\mathcal J_{\varepsilon_k}(P_{h_k}U^\ell(Z_k))-\mathcal J_{\varepsilon_k}(P_{h_k}I_s(Z_k))|\,\dd s\right]
\\
&\le C\sum_{j=1}^2\int_{t_{\ell-1}}^{t_\ell}\Bbb E_{\tau_k}\left[\|P_{h_k}U^\ell(Z_k)-P_{h_k}I_s(Z_k)\|^j_{\Bbb H^1_0}\right]\,\dd s
\\
&+ C \sum_{j=1}^2\int_{t_{\ell-1}}^{t_\ell}\Bbb E_{\tau_k}\left[\|P_{h_k}U^\ell(Z_k)-P_{h_k}I_s(Z_k)\|^\frac{j}{2}_{\Bbb H^1_0}\|P_{h_k}U^\ell(Z_k)\|^\frac{\ell}{2}_{\Bbb H^1_0}\right]\,\dd s
\\
&\le C\tau_k^{1+\frac{\theta}{2}}\,,
\end{align*}
by the stability of the projections $\{P_h\}_{h>0}$ in $\Bbb H^1_0$ from Assumption \ref{ass_vh}$_{(3)}$.

Now denote by $R_k$ the set of $\ell\in\{1,\dots,N_k\}$ such that the interval $(t_{\ell-1},t_\ell)$ is not fully contained in some of the intervals $(s_\alpha,s_{\alpha+1}]$ for $\alpha\in\{0,\dots,m-1\}$. If $\ell\in R_k$ then there exists unique $\alpha$ such that $s_\alpha<t_\ell\le s_{\alpha+1}$. If $s_\alpha\le t_{\ell-1}$ then this would contradict that $\ell\in R_k$ hence $s_\alpha<t_\ell<s_\alpha+\tau_k$. In particular, $\operatorname{card}\,(R_k)\le m$, and consequently
\begin{align*}
\sum_{\ell=1}^{N_k}\Bbb E_{\tau_k}&\left[\int_{t_{\ell-1}}^{t_\ell}|(P_{h_k}G_{t_\ell}(Z_k)-P_{h_k}G_s(Z_k),S^1_s(Z_k)-U^\ell(Z_k))|\,\dd s\right]
\\
&=\sum_{\ell\in R_k}\Bbb E_{\tau_k}\left[\int_{t_{\ell-1}}^{t_\ell}|(P_{h_k}G_{t_\ell}(Z_k)-P_{h_k}G_s(Z_k),S^1_s(Z_k)-U^\ell(Z_k))|\,\dd s\right]
\\
&\le C m \tau_k\,.
\end{align*}
Analogously, we estimate
$$
\sum_{\ell=2}^{N_k}\Bbb E_{\tau_k}\left[\int_{t_{\ell-2}}^{t_{\ell-1}}\|B(S^1_s(Z_k))-H_{t_{\ell-1}}(Z_k)\|_{\mathscr L_2(\ell_2,\Bbb L^2)}^2\,\dd s\right] \leq C m \tau_k\,,
$$
$$
\sum_{\ell=2}^{N_k}\Bbb E_{\tau_k}\left[\int_{t_{\ell-2}}^{t_{\ell-1}}\|B(S^1_s(Z_k))-H_s(Z_k)\|_{\mathscr L_2(\ell_2,\Bbb L^2)}^2\,\dd s\right]  \leq C m \tau_k\,,
$$
by the linear growth of $B$ assumed in ${(\mathbf B_1)}$. In the fourth step, we estimate
$$
\Bbb E_{\tau_k}\left[\int_{t_{\ell-1}}^{t_\ell}|(P_{h_k}G_s(Z_k),U^\ell(Z_k))-(P_{h_k}G_s(Z_k),I_s(Z_k))|\,\dd s\right]\le\mathbf C\tau_k^{1+\theta}\,,
$$
by boundedness of $G$.} Hence, we conclude that
\begin{align}\label{disc5}
\nonumber
\frac{1}{2}\Bbb E_{\tau_k}[\|S^1_{t_i}(Z_k)&-P_{h_k}I_{t_i}(Z_k)\|^2]+\Bbb E_{\tau_k}\left[\int_0^{t_i}\mathcal J(S^1_s(Z_k))\,\dd s\right]\le
\\
\frac{1}{2}\Bbb E_{\tau_k}[\|S^1_{t_i}(Z_k)&-P_{h_k}I_{t_i}(Z_k)\|^2]+\Bbb E_{\tau_k}\left[\int_0^{t_i}\mathcal J_{\varepsilon_k}(S^1_s(Z_k))\,\dd s\right]\le\frac{1}{2}\|x^0-u^0\|^2
\\
\nonumber
&+\Bbb E_{\tau_k}\left[\int_0^{t_i}[\mathcal J_{\varepsilon_k}(P_{h_k}I_s(Z_k))+(P_{h_k}G_s(Z_k),S^1_s(Z_k)-I_s(Z_k))]\,\dd s\right]
\\
\nonumber
&+\frac{1}{2}\Bbb E_{\tau_k}\left[\int_0^{t_i}\|B(S^1_s(Z_k))-H_s(Z_k)\|_{\mathscr L_2(\ell_2,\Bbb L^2)}^2\,\dd s\right]+C\tau_k^{{\frac{\theta}2}}\,,
\end{align}
for $0\le i\le N_k$ and some $C$ independent of $i$ and $k$. Here we used $\mathcal J \leq \mathcal J_\eps$ and the linear growth of $B$ assumed in ${(\mathbf B_1)}$.

\medskip

{\bf(ii)} In the second step, we extend the discrete result from step {\bf (i)} to the time-continuous case {on the stochastic basis $(\mathbf Z,\mathscr B(\mathbf Z),(\mathcal Z^\mu_t),\mu)$}, yet still for the {simple}
processes $G$ and $H$ defined in part {\bf(i)}.

We note that by construction the mapping $I:[0,T]\times\mathbf Z\to\Bbb H^1_0$ from {\bf(i)} is continuous and the following properties hold for every $k$ and $r\in[0,T]$:
\begin{itemize}
\item[(a)] $\|S^1_r-P_{h_k}I_r\|^2$ is lower semicontinuous on $\mathbf Z$,
\item[(b)] $\int_0^r\mathcal J(S^1)\,\dd s$ is lower semicontinuous on $\mathbf Z$ by Remark \ref{funct_i_sup},
\item[(c)] $\int_0^r[\mathcal J_{\varepsilon_k}(P_{h_k}I)+(P_{h_k}G,S^1-I)]\,\dd s$ is continuous on $\mathbf Z$ as $(P_{h_k}G,S^1)=(G,P_{h_k}S^1)$,
\item[(d)] $\int_0^r\|B(S^1)-H\|_{\mathscr L_2(\ell_2,\Bbb L^2)}^2\,\dd s$ is $\mathscr B(\mathbf Z)$-measurable by Corollary \ref{first_prop} (I).
\end{itemize} 
Furthermore, from the fact that $\mathcal J_{\varepsilon}\rightarrow \mathcal J$ for $\eps \rightarrow 0$ and $\|\nabla v - \nabla P_{h}v\|\rightarrow 0$, $v\in \mathbb{H}^1$ for $h\rightarrow 0$ we deduce the convergence
\begin{align*}
\|S^1_t(z)-I_t(z)\|^2&\le\liminf_{k\to\infty}\|S^1_{t^k_{i_k}}(z_k)-P_{h_k}I_{t^k_{i_k}}(z_k)\|^2\,,
\\
\int_0^r\mathcal J(I(z))\,\dd s=&\lim_{k\to\infty}\int_0^r\mathcal J_{\varepsilon_k}(P_{h_k}I(z_k))\,\dd s\,,
\\
\int_0^r(G(z),S^1(z)-I(z))\,\dd s=&\lim_{k\to\infty}\int_0^r(P_{h_k}G(z_k),S^1(z_k)-I(z_k))\,\dd s\,,
\\
\int_0^r\|B(S^1(z))-H(z)\|_{\mathscr L_2(\ell_2,\Bbb L^2)}^2\,\dd s=&\lim_{k\to\infty}\int_0^r\|B(S^1(z_k))-H(z_k)\|_{\mathscr L_2(\ell_2,\Bbb L^2)}^2\,\dd s\,,
\end{align*}
whenever $t^k_{i_k}\nearrow t$ and $z_k\to z$ in the sense of (i)-(iv) of Lemma \ref{cont_meas} where, in the last step, we used the assumption ${(\mathbf B_2)}$ on continuity of $B$ if $d\ge 2$ (if $d=1$, continuity of $B$ suffices). Indeed, assume that 
\begin{equation}\label{mean_vonv_b}
\int_0^T\|B(f_k)-B(f)\|^2_{\mathscr L_2(\ell_2,\Bbb L^2)}\,\dd s\ge r>0\,,
\end{equation}
for some $f_k\to f$ in the sense of (i) of Lemma \ref{cont_meas}. Then $f_k\to f$ uniformly in $\Bbb H^{-d}$ and $\int_0^T\|f_k\|_{BV}\,ds\le C$. Hence $\int_0^T\|f_k-f\|_{\Bbb L^1}\,ds\to 0$ 
since $BV(\mathcal O)\hookrightarrow\hookrightarrow\Bbb L^1\hookrightarrow\Bbb H^{-d}$. 
If $d=1$ then even $\int_0^T\|f_k-f\|_{\Bbb L^2}\,ds\to 0$ since $BV(\mathcal O)\hookrightarrow\hookrightarrow\Bbb L^2\hookrightarrow\Bbb H^{-d}$. Thus there exists a subsequence $k_l$ such that $f_{k_l}(s)\to f(s)$ a.e. on $\mathcal O$ (or in $\Bbb L^2$ if $d=1$) for a.e. $s\in[0,T]$. In particular, {$\|B(f_{k_l}(s))-B(f(s))\|_{\mathscr L_2(\ell_2,\Bbb L^2)}\to 0$} for a.s. $s\in[0,T]$, and the linear growth of $B$ then yields that $\int_0^T\|B(f_{k_l})-B(f)\|^2_{\mathscr L_2(\ell_2,\Bbb L^2)}\,\dd s\to 0$ which is a contradiction with \eqref{mean_vonv_b}. Finally,
\begin{align*}
\|I_s(Z_k)\|_{\Bbb H^1_0}&\le c+c\sum_{j=1}^{j_0}\|W_{\tau_k}^j\|_{C([0,T])}\,,
\\
|\mathcal J_{\varepsilon_k}(P_{h_k}(I_s(Z_k)))|&\le c[1+\|I_s(Z_k)\|^2_{\Bbb H^1_0}]\,,
\end{align*}
holds by stability of the projections $\{P_h\}_{h> 0}$ in $\Bbb H^1_0$ so \eqref{unif_int_fk} is satisfied by \eqref{donsker_besov}, Lemma \ref{un_est} and the linear growth of $B$. Hence 
on taking the limit $k\rightarrow \infty$ in (\ref{disc5}) we conclude by Lemma \ref{cont_meas} that \eqref{SVI_def} holds.

\medskip

{\bf(iii)} In the last step, we prove the full result. The extension of \eqref{SVI_def} to $(\mathcal Z^\mu_t)$-progressively measurable processes in $L^2([0,T]\times\Omega;\Bbb H^1_0)$ and $L^2([0,T]\times\Omega;\mathscr L_2(\ell_2,\Bbb H^1_0))$ goes via a standard density argument, and the general case can be obtained by considering $G_h=P_hG$ and $H_h=P_hH$, and then letting $h\to 0$.
\end{proof}

{\section{Convergence to pathwise unique probabilistically strong solution}\label{CoStSo}

In this section we study convergence of the interpolants $\overline{X}_{\varepsilon,h,\tau}$, $\underline{X}_{\varepsilon,h,\tau}$ and $X_{\varepsilon,h,\tau}$ to 
a probabilistically strong SVI solution of \eqref{TVF} in probability.

\begin{thm}\label{strong_exist}
Let $(W^j)$ be independent $(\mathcal F_t)$-Wiener processes on $(\Omega,\mathcal F,(\mathcal F_t),\Bbb P)$ and let
$$
\xi_\tau^{i,j}=W^j(t_i)-W^j(t_{i-1}),\qquad t_i=i\tau.
$$
Assume also that pathwise uniqueness holds for the SVI solutions of \eqref{TVF} satisfying \eqref{energ_eq}. Then $\overline{X}_{\varepsilon,h,\tau}$, $\underline{X}_{\varepsilon,h,\tau}$ and $X_{\varepsilon,h,\tau}$ converge to $X$ in probability in $Q_{c,BV}([0,T];\Bbb L^2_w)$, $Q_c([0,T];\Bbb L^2_w)$ and $C([0,T];\Bbb L^2_w)$ respectively {where $X$ is a solution \eqref{TVF} with respect to $(W^j)_{j\in\mathbb{N}}$}.
\end{thm}

\begin{proof}
The proof is based on the Gyongy-Krylov Lemma 1.1 in \cite{Gyongy_Krylov}. Define
$$
\mathbf S=Q_{c,BV}([0,T];\Bbb L^2_w)\times Q_{c,BV}([0,T];\Bbb L^2_w)\times C([0,T])\times C([0,T])\times C([0,T])\times C([0,T])\times\dots,
$$
and the projections
\begin{align*}
\nonumber
&Y^1:\mathbf S\to Q_{c,BV}([0,T];\Bbb L^2_w)&(f^1,f^2,w^1,w^2,w^3,\dots)\mapsto f^1,
\\ \nonumber
&Y^2:\mathbf S\to Q_{c,BV}([0,T];\Bbb L^2_w)&(f^1,f^2,w^1,w^2,w^3,\dots)\mapsto f^2,
\\ 
&W^j:\mathbf S\to C([0,T])&(f^1,f^2,w^1,w^2,w^3,\dots)\mapsto w^j,
\end{align*}
and the canonical filtration on $\mathbf S$
$$
\mathcal S_t=\sigma(Y^1_s,Y^2_s,W^j_s:\,s\in[0,t],\,j\in\Bbb N),\qquad t\in[0,T].
$$
We consider two different sequences of discretization parameters $(\varepsilon^i_k,h^i_k,\tau^i_k)\to(0,0,0)$ for $i=1,2$, which are chosen as in Corollary~\ref{main_cor}, 
such that 
$$
Z_k:=(\overline{X}_{\varepsilon^1_k,h^1_k,\tau^1_k},\overline{X}_{\varepsilon^2_k,h^2_k,\tau^2_k},W^1,W^2,W^3,\dots),
$$
converge to a Radon probability measure $\theta$ on $\mathscr B(\mathbf S)$. Analogically as in Corollary \ref{first_prop}, the processes $Y^1$, $Y^2$ and $\Wproc$ are $(\mathcal S_t)$-progressively measurable, paths of $Y^1$ and $Y^2$ are continuous $\theta$-a.s.,
$$
\int_{\mathbf S}\left[\sup_{t\in[0,T]}\|Y^i(t)\|^4+\left(\int_0^T\|Y^i(t)\|_{BV(\mathcal O)}\,\dd t\right)^2\right]\,\dd\theta<\infty\,,\qquad i=1,2,
$$
the $\sigma$-algebras $\mathcal S_t$ and $\sigma(W^j(b)-W^j(a):\,t\le a\le b\le T,\,j\in\Bbb N)$ are $\theta$-independent and $W^1,W^2,W^3,\dots$ are $\theta$-independent $(\mathcal S_t)$-Brownian motions. The proof that $Y^1$ and $Y^2$ are SVI solutions with respect to $\Wproc$ and 
$$
\theta\,[Y^1(0)=Y^2(0)=x^0]=1
$$
is analogous to the proof of Theorem \ref{main_full_1}, we point out the differences below. 

In step {\bf(i)} one modifies the definition of the $\Bbb R^{3M^2}$-valued random variables
$$
V^\alpha=((\varphi_\beta, Y^1_{r^{\alpha}_\gamma}),(\varphi_\beta,Y^2_{r^{\alpha}_\gamma}),W^j_{r^{\alpha}_\gamma}:\beta,\gamma,j\in\{1,\dots,M\}),\quad 0\le\alpha\le m\,,
$$
defined on $\mathbf S$, the functions $g_\alpha,h_{\alpha,j}$ map $\Bbb R^{3M^2}$ to $\Bbb H^1_0$ and have the same properties as in the proof of Theorem \ref{main_full_1} and
$$
G(t)=\sum_{\alpha=1}^{m-1}\mathbf 1_{(s_\alpha,s_{\alpha+1}]}(t)g_\alpha({V^{\alpha-1}}),\qquad H_j(t)=\sum_{\alpha=1}^{m-1}\mathbf 1_{(s_\alpha,s_{\alpha+1}]}(t)h_{\alpha,j}({V^{\alpha-1}})\,,
$$
i.e., there is a backward time shift compared to the definition of $G$ and $H$ in the proof of Theorem \ref{main_full_1}. 
Once we set we set $N^i_k=T/\tau^i_k$, $t^i_\ell:=\ell\tau^i_k$ for $\ell\in\{0,\dots,N^i_k\}$, $i=1,2$  the above modification
ensures that $V^{\alpha-1}(Z_k)$ is $\mathcal F_{s_\alpha}$-measurable. Consequently, $G(t,Z_k)$ and $H_j(t,Z_k)$ are $(\mathcal F_t)$-adapted processes as long as $\tau^i_k$, $i=1,2$ are smaller than the mesh of the partition $\{s_\alpha\}$.

Pathwise uniqueness of solutions of \eqref{TVF} yields that $Y^1=Y^2$ holds $\mathbb{P}$-a.s. hence $\overline X_{\varepsilon,h,\tau}$ is convergent in $Q_{c,BV}([0,T];\Bbb L^2_w)$ in probability as $(\varepsilon,h,\tau)\to (0,0,0)$
by \cite[Theorem 2.10.3]{Breit_Feireisl_Hofmanova} (with the exception that we apply the Gyongy-Krylov lemma directly without having to pass to a subsequence as in \cite[Theorem 2.10.3]{Breit_Feireisl_Hofmanova}).

Now we apply the Gyongy-Krylov lemma once again. By Corollary \ref{main_cor} we deduce that the laws of the sequence 
$$
(\overline{X}_{\varepsilon^1_k,h^1_k,\tau^1_k},\underline{X}_{\varepsilon^1_k,h^1_k,\tau^1_k},X_{\varepsilon^1_k,h^1_k,\tau^1_k},\overline{X}_{\varepsilon^2_k,h^2_k,\tau^2_k},\underline{X}_{\varepsilon^2_k,h^2_k,\tau^2_k},X_{\varepsilon^2_k,h^2_k,\tau^2_k}),
$$
on 
$$
\mathscr B(Q_{c,BV}\times Q_c\times C\times Q_{c,BV}\times Q_c\times C),
$$
where $Q_{c,BV}=Q_{c,BV}([0,T];\Bbb L^2_w)$, $Q_c=Q_c([0,T];\Bbb L^2_w))$ and $C=C([0,T];\Bbb L^2_w)$ converge to some probability measure $\nu$. Consequently
$$
\nu\{x_1=x_2=x_3,\,x_4=x_5=x_6,\,x_1=x_4\}=1,
$$
by Corollary \ref{first_prop} (II) and the first part of the proof. Hence \cite[Theorem 2.10.3]{Breit_Feireisl_Hofmanova} yields that $\underline{X}_{\varepsilon,h,\tau}$ and $X_{\varepsilon,h,\tau}$ converge in probability in $Q_c([0,T];\Bbb L^2_w))$ and $C([0,T];\Bbb L^2_w)$ respectively as $(\varepsilon,h,\tau)\to(0,0,0)$. And the limit equals to $X$ by \eqref{un_est_3}.

{Analogously as in the proof of Theorem \ref{main_full_1} we set 
\begin{equation}\label{def_G_H_2}
G(t)=\sum_{\alpha=0}^{m-1}\mathbf 1_{(s_\alpha,s_{\alpha+1}]}(t)g_\alpha \qquad 
H_j(t)=\sum_{\alpha=0}^{m-1}\mathbf 1_{(s_\alpha,s_{\alpha+1}]}(t)h_{\alpha,j},
\end{equation}
for some $0=s_0<\dots<s_m=T$ where $g_\alpha$ and $h_{\alpha,j}$ are simple $\Bbb H^1_0$-valued $\mathcal F_{s_\alpha}$-measurable random variables such that $h_{\alpha,j}=0$ for $j\ge j_0$ for some arbitrary $j_0\in\Bbb N$ and define the process $I$ as in \eqref{def_I}. Setting $N=T/\tau$, $t_i:=i\tau$ for $i\in\{0,\dots,N\}$ then, as in \eqref{disc5} in the proof of Theorem \ref{main_full_1} we obtain that
\begin{align*}
\frac{1}{2}\Bbb E[\|\overline X_\tau(t_i)&-P_hI(t_i)\|^2]+\Bbb E\left[\int_0^{t_i}\mathcal J(\overline X_\tau(s))\,\dd s\right]\le\frac{1}{2}\|x^0-u^0\|^2
\\
&+\Bbb E\left[\int_0^{t_i}[\mathcal J_\varepsilon(P_hI(s))+(P_hG(s),\overline X_\tau(s)-I(s))]\,\dd s\right]
\\
&+\frac{1}{2}\Bbb E\left[\int^{t_i}_0\|B(\overline X_\tau(s))-H(s)\|_{\mathscr L_2(\ell_2,\Bbb L^2)}^2\,\dd s\right]+c\tau^{\frac{\theta}{2}},
\end{align*}
for $0\le i\le N$ if $N\ge j_0$. If $0\le t\le t_i<t+\tau$ then
\begin{align*}
\frac{1}{2}\Bbb E[\|\overline X_\tau(t_i)&-P_hI(t_i)\|^2]+\Bbb E\left[\int_0^t\mathcal J(\overline X_\tau(s))\,\dd s\right]\le\frac{1}{2}\|x^0-u^0\|^2
\\
&+\Bbb E\left[\int_0^t[\mathcal J_\varepsilon(P_hI(s))+(P_hG(s),\overline X_\tau(s)-I(s))]\,\dd s\right]
\\
&+\frac{1}{2}\Bbb E\left[\int^t_0\|B(\overline X_\tau(s))-H(s)\|_{\mathscr L_2(\ell_2,\Bbb L^2)}^2\,\dd s\right]+c\tau^{\frac{\theta}{2}}+c_1\tau.
\end{align*}
We deduce that the following holds for $(\varepsilon,h,\tau)\to(0,0,0)$:
\begin{itemize}
\item $\overline X_\tau(t_i)-P_hI(t_i)$ is tight in $\Bbb L^2_w$ and converges to $X(t)-I(t)$ in $\Bbb L^2_w$ in probability (hence also in law) thus
$$
\Bbb E[\|\overline X(t)-(t)\|^2]\le\liminf\Bbb E[\|\overline X_\tau(t_i)-P_hI(t_i)\|^2],
$$
by Proposition \ref{portmanteau},
\item $\overline X_\tau$ converges to $X$ in $Q_{c,BV}([0,T];\Bbb L^2_w)$ in probability (hence also in law) thus
$$
\Bbb E\left[\int_0^t\mathcal J(X(s))\,\dd s\right]\le\liminf\Bbb E\left[\int_0^t\mathcal J(\overline X_\tau(s))\,\dd s\right],
$$
as in the proof of Theorem \ref{main_full_1},
\item 
$$
\Bbb E\left[\int_0^t\mathcal J(I(s))\,\dd s\right]=\lim\Bbb E\left[\int_0^t\mathcal J_\varepsilon(P_hI(s))\,\dd s\right],
$$
as in the proof of Theorem \ref{main_full_1},
\item
$$
\Bbb E\left[\int_0^T|(P_hG(s)-G(s),\overline X_\tau(s)-I(s))|\,\dd s\right]\le C\Bbb E\left[\int_0^T\|P_hG(s)-G(s)\|^2\,\dd s\right]\to 0\,,
$$
\item
\begin{align*}
\Bbb E\left[\int_0^T|(G(s),\overline X_\tau(s)-X(s))|\,\dd s\right]&=\sum_{\alpha=0}^{m-1}\Bbb E\left[\int_{s_{\alpha}}^{s_{\alpha+1}}|(g_\alpha,\overline X_\tau(s)-X(s))|\,\dd s\right]
\\
&\le T\sum_{\alpha=0}^{m-1}\Bbb E\left[\sup_{s\in[0,T]}|(g_\alpha,\overline X_\tau(s)-X(s))|\right]\to 0,
\end{align*}
since $g_\alpha$ are simple,
\item we proved in the proof of Theorem \ref{main_full_1} that if $f^j_n\to f^j$ in $Q_{c,BV([0,T];\Bbb L^2_w)}$ and 
$$
\int_0^T\|f^j_n(s)\|_{BV(\mathcal O)}\,\dd s\le C,\qquad j=1,2,
$$ 
then
$$
\int_0^T\|B(f^1_n)-B(f^2_n)\|_{\mathscr L_2(\ell_2,\Bbb L^2)}^2\,\dd s\to\int_0^T\|B(f^1)-B(f^2)\|_{\mathscr L_2(\ell_2,\Bbb L^2)}^2\,\dd s.
$$
Now  $(\overline X_\tau,X)$ are tight in $Q_{c,BV([0,T];\Bbb L^2_w)}\times Q_{c,BV([0,T];\Bbb L^2_w)}$ and converge in probability (hence in law) to $(X,X)$. By Proposition~\ref{portmanteau}
we deduce
$$
\lim\Bbb E\left[\int^T_0\|B(\overline X_\tau(s))-B(X(s))\|_{\mathscr L_2(\ell_2,\Bbb L^2)}^2\,\dd s\right]=\Bbb E\left[\int^T_0\|B(X(s))-B(X(s))\|_{\mathscr L_2(\ell_2,\Bbb L^2)}^2\,\dd s\right].
$$
\end{itemize}
Hence, we obtain as in {\bf (ii)} in the proof of Theorem \ref{main_full_1} that
\begin{align*}
\frac{1}{2}\Bbb E[\|X(t)&-I(t)\|^2]+\Bbb E\left[\int_0^t\mathcal J(X(s))\,\dd s\right]\le\frac{1}{2}\|x^0-u^0\|^2
\\
&+\Bbb E\left[\int_0^t[\mathcal J(I(s))+(G(s),X(s)-I(s))]\,\dd s\right]
\\
&+\frac{1}{2}\Bbb E\left[\int^t_0\|B(X(s))-H(s)\|_{\mathscr L_2(\ell_2,\Bbb L^2)}^2\,\dd s\right]\,.
\end{align*}

The extension to general $G$, $H$ and $I$ is analogous to {\bf (iii)} in the proof of Theorem \ref{main_full_1}.
}

\end{proof}
}
}

\section{Numerical experiments}\label{sec_numexp}

We perform numerical experiments using a generalization of the fully discrete finite element scheme (\ref{num_tvf})
with $\mathcal{O} = (0,1)^2$. We consider a triangulation $\mathcal{T}_h$ of $\mathcal{O}$ for $h=2^{-\ell}$ 
which is obtained by subdividing the unit square into sub-squares of size $h$ and subsequently each square is subdivided into four equal right-angled triangles.
Given $\mathbb{V}_h \equiv \mathbb{V}_h(\mathcal{T}_h) = \mathrm{span}\{\phi_j, \,\, j=1,\dots,J\}$ and a constant $\sigma>0$ we set $B_j = \sigma\phi_j$ and denote $\Delta_i W_h = \sum_{j=1}^J \phi_j \Delta_i \beta_j$ with
discrete increments $\Delta_i \beta_j := \beta_j(t_i)-\beta_j(t_{i-1})$ and
where ${\beta}_j$, $j=1,\dots, J$ are independent scalar-valued Wiener processes.

The corresponding counterpart of the scheme (\ref{num_tvf}) for $i=1, \dots, N$ then reads as
\begin{align}\label{fem_scheme}
\ska{X^i_{\varepsilon,h},\vh} =& \ska{X^{i-1}_{\varepsilon,h},\vh}-\tau \ska{\frac{\nabla X^i_{\varepsilon,h}}{\sqrt{|\nabla X^i_{\varepsilon,h}|^2+\eps^2}},\nabla\vh } \nonumber \\
 & -\tau\lambda\ska{X^i_{\varepsilon,h} - g_h,\vh}+ \sigma \ska{ {\Delta_i W_h},\vh} &&\forall \vh \in \mathbb{V}_h\,, 
\\ \nonumber
X^0_{\varepsilon,h}  = & x^0_h\,.
\end{align}
where $g_h,\,x^0_h\in\widetilde{\mathbb{V}}_h\subset \mathbb{V}_h$ are suitable approximations (see below) of the data $g$, $x_0$, respectively.

For comparison we also perform simulations using a non-conforming variant of (\ref{fem_scheme})
where the ($\mathbb{H}^1$-conforming) space $\mathbb{V}_h$ in (\ref{fem_scheme}) is replaced by a non-conforming finite element space $\mathbb{V}_{\mathrm{cr}}\not\subset \mathbb{H}^{1}_0$.
Given a partition $\mathcal{T}_h$ of $\mathcal{O}$ we denote the set of all faces of elements $T\in\mathcal{T}_h$ as $\mathcal{S}_h = \cup_{T\in\mathcal{T}_h}\partial T$
and for a face $S\in \mathcal{S}_h$ we denote its barycenter by $b_S$.
Then we define the non-conforming finite element space as
\begin{align*}
\mathbb{V}_{\mathrm{cr}} = \big\{& \varphi\in \mathbb{L}^2;\,\, \varphi|_T \in \mathcal{P}^1(T) \,\, \forall{T}\in\mathcal{T}_h,\,\,\varphi \text{ is continuous at } b_S\,\, 
\forall S\in\mathcal{S}_h\cap\mathcal{O}
\\ 
& \text{ and } { \varphi(b_S) = 0 \text{ for } S\in\mathcal{S}_h\cap\partial\mathcal{O} }\big\}\,.
\end{align*}
The above finite element space corresponds to the first order Crouzeix-Raviart finite element which is more suitable for the approximation
of discontinuous solutions, cf., \cite{bartels_cr} and the references therein, for its use in the context of image processing.
We note that $\mathbb{V}_{h}(\mathcal{T}_h) \subset \mathbb{V}_{\mathrm{cr}}(\mathcal{T}_h)$ but since $\mathbb{V}_{\mathrm{cr}} \not\subset \mathbb{H}^{1}_0$
the  elements of $\mathbb{V}_{\mathrm{cr}}$ have no (global) weak gradients in general.
Hence for $w_h \in \mathbb{V}_{\mathrm{cr}}$ we define a discrete gradient $\nabla_h w_h$ 
via $\nabla_h w_h = \nabla (w_h|_T)$. Then the non-conforming counterpart of the scheme (\ref{num_tvf}) is obtained by replacing $\mathbb{V}_h$
with $\mathbb{V}_{\mathrm{cr}}$ and the gradients $\nabla$ in (\ref{num_tvf}) by the discrete gradient $\nabla_h$. 
The numerical solutions $X^i_{\varepsilon,h}\in \mathbb{V}_{\mathrm{cr}}$, $i=1,\dots, N$ the exist and satisfy an energy law (counterpart of Lemma~\ref{lem_energy},
however, the convergence of the non-conforming scheme is open so far.


To construct an approximation of the data $g$, $x_0$ we consider the space $\widetilde{\mathbb{V}}_h\equiv \mathbb{V}_h(\widetilde{\mathcal{T}}_h)$ with fixed mesh size $h=2^{-6}$.
We define the exact ``image'' $\tilde{g}_h\in \widetilde{\mathbb{V}}_h$ as the composition of the characteristic function of a square with side $\frac{1}{2}$ at the center of $\mathcal{O}$
scaled by the factor $\frac{1}{2}$ and the characteristic function of a circle with radius $\frac{1}{4}$ shifted by $0.2$ to the right of the center of $\mathcal{O}$
interpolated on the mesh $\widetilde{\mathcal{T}}_h$, see Figure~\ref{fig_data} (left), i.e., $\displaystyle \tilde{g}_h(x) = \sum_{j=1}^{\tilde J} \tilde{g}(x_j) \tilde\phi_j (x)$ 
where $\{\tilde\phi_j\}_{j=1}^{\tilde J}$  are the nodal basis functions associated with the nodes the $\{x_j\}_{j=1}^{\tilde J}$ of the mesh $\widetilde{\mathcal{T}}_h$.
Hence, we set $g_h = \tilde{g}_h + \xi_h\in \widetilde{\mathbb{V}}_h$ with the ``noise'' $\displaystyle \xi_h(x) = 0.1\sum_{j=1}^{\tilde J} \tilde\phi_j (x) \xi_\ell$, $x\in\mathcal{O}$
where $\xi_j$, $j=1,\dots, \tilde J$ are realizations of independent $\mathcal{U}(-1,1)$-distributed random variables.
The corresponding realization of the noise $\xi_h$ and the resulting ``noisy image'' $g_h$ are displayed in Figure~\ref{fig_data} (middle and right, respectively).

\begin{figure}[!htp]
\center
\includegraphics[width=0.3\textwidth]{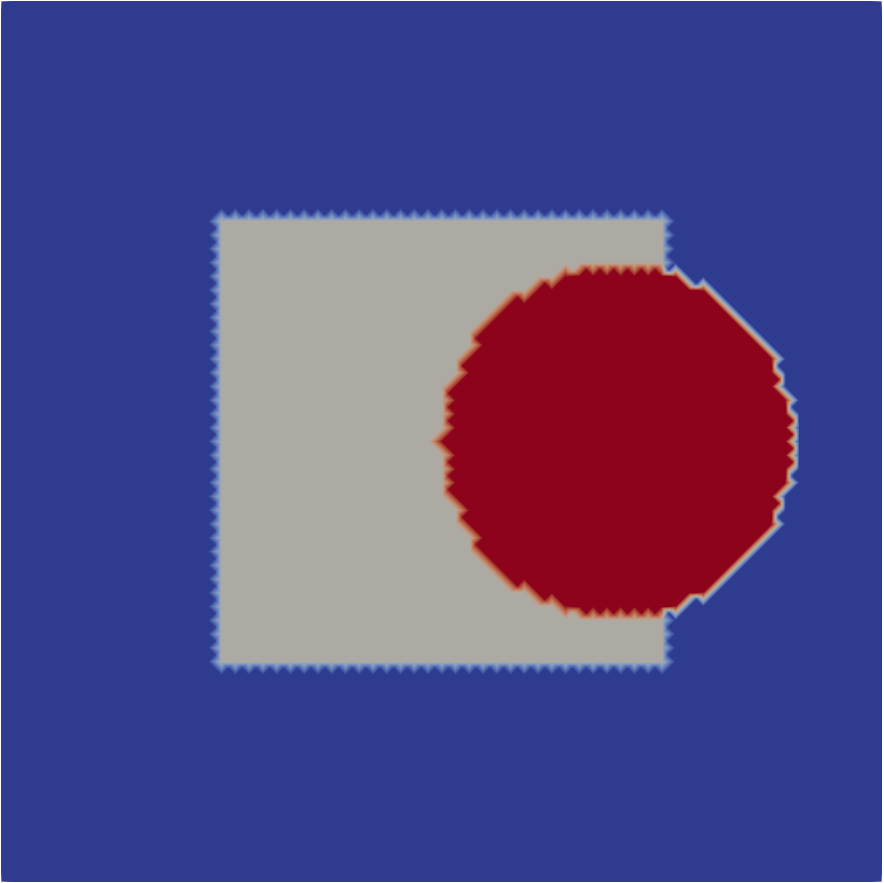}
\includegraphics[width=0.3\textwidth]{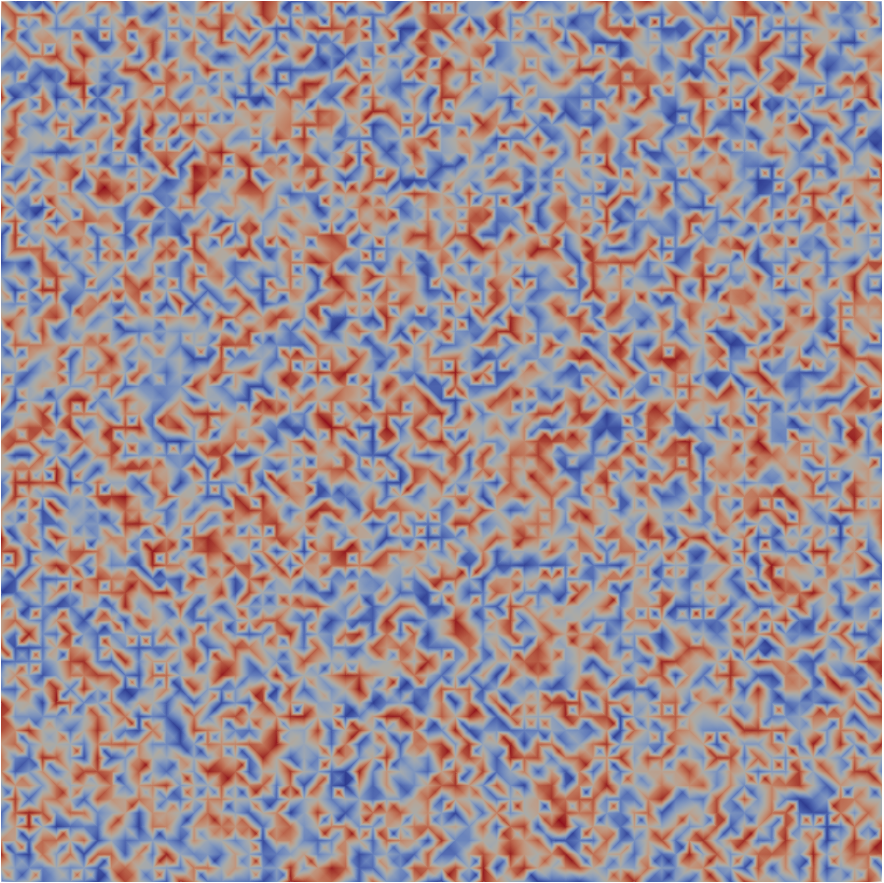}
\includegraphics[width=0.3\textwidth]{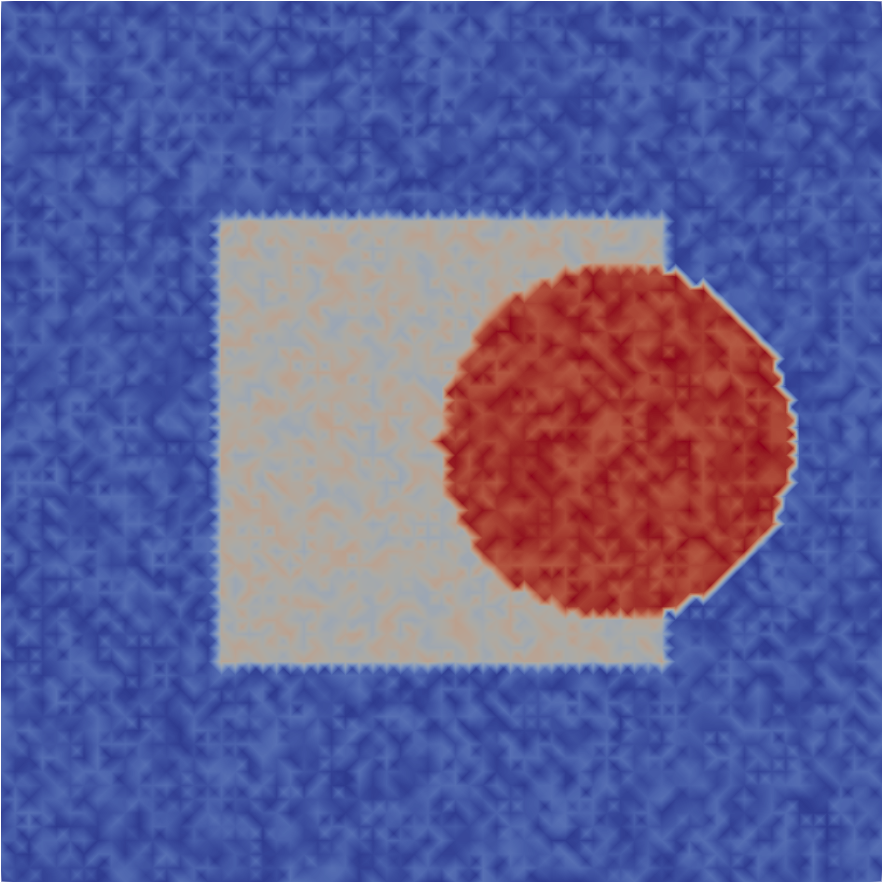}
\caption{The original image $\tilde{g}_h$ (left), the noise $\xi_h$ (middle) and the noisy image $g_h$ (right).}
\label{fig_data}
\end{figure}

In all experiments we set $T=0.1$, $\lambda=200$, $\eps=10^{-4}$, $x^0_h = \tilde{g}_h$.
The nonlinear algebraic system which corresponds to (\ref{fem_scheme}) is solved using
a simple fixed-point iterative scheme with tolerance $10^{-4}$.
If not mentioned otherwise we use the time step $\tau = 10^{-3}$, the mesh size $h = 2^{-6}$ and $\sigma=1$.

The time-evolution of the discrete energy functional $\mathcal{J}_{\eps}$
for one realization of the space-time noise $W_h$ is displayed in Figure~\ref{fig_ener} (left); $\mathrm{P1}$ denotes the solution with the conforming finite element approximation
and $\mathrm{CR}$ denotes the non-conforming approximation, $\mathrm{h7}$, $\mathrm{h8} $ respectively denote the solution with mesh size $h=2^{-7},\,\, 2^{-8}$ and $det$ stands for the deterministic solution with $\sigma=0$.
The evolution of the approximation error of the original image $\tilde{g}_h$ is displayed in Figure~\ref{fig_ener} (left).
We make the following observations for the conforming finite element method: the approximation error for $\sigma=0$ improves with decreasing mesh size,
and the approximation error of the stochastic problem oscillates around the error of the deterministic counterpart.
For the non-conforming approximation we measure the approximation error of the projected discrete solution $\Pi_h^0 X_{\eps,h}^i$, $i=1,\dots$,
where $\Pi_h^0$ is the projection onto piecewise constant functions on $\mathcal{T}_h$, see Figure~\ref{fig_finalu} where we also display the 
solution of the conforming finite element scheme.
As expected, cf. \cite{bartels_cr}, on the same mesh with $\sigma=0$ the non-conforming finite element method yields a better approximation of the original image
then the conforming method. The non-conforming approximation requires roughly $3\times$ more degrees of freedom than the conforming one
but the approximation is still comparable to the conforming method with smaller mesh size $h=2^{-7}$ (which involves $4\times$ more degrees of freedom than the approximation with $h=2^{-6}$).
Nevertheless, we also observe that the non-conforming approximation is more sensitive to the noise.
For comparison in Figure~\ref{fig_finalu_pw} we display the piecewise constant projections of the solutions computed with the conforming scheme with $h=2^{-6}$ and $h=2^{-8}$.

\begin{figure}[!htp]
\center
\includegraphics[width=0.49\textwidth]{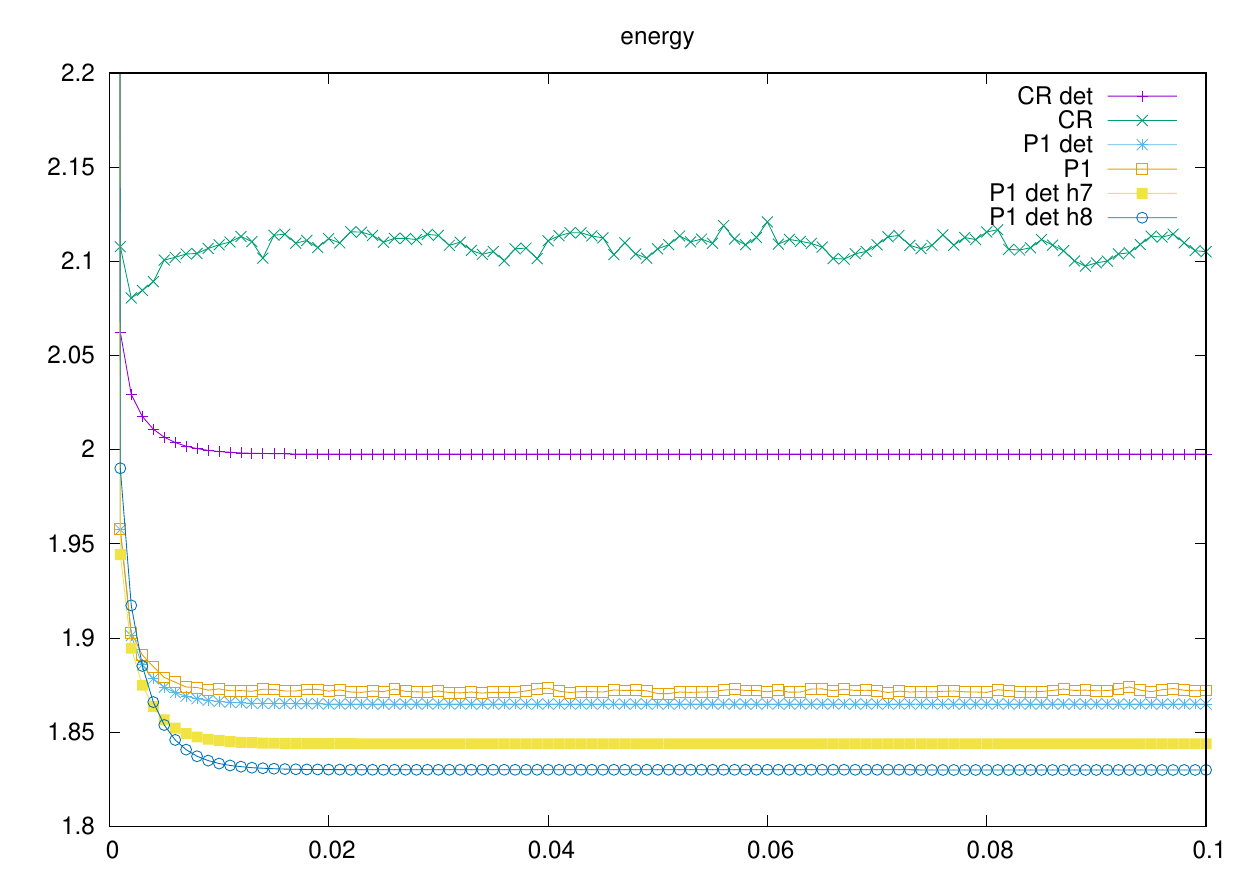}
\includegraphics[width=0.49\textwidth]{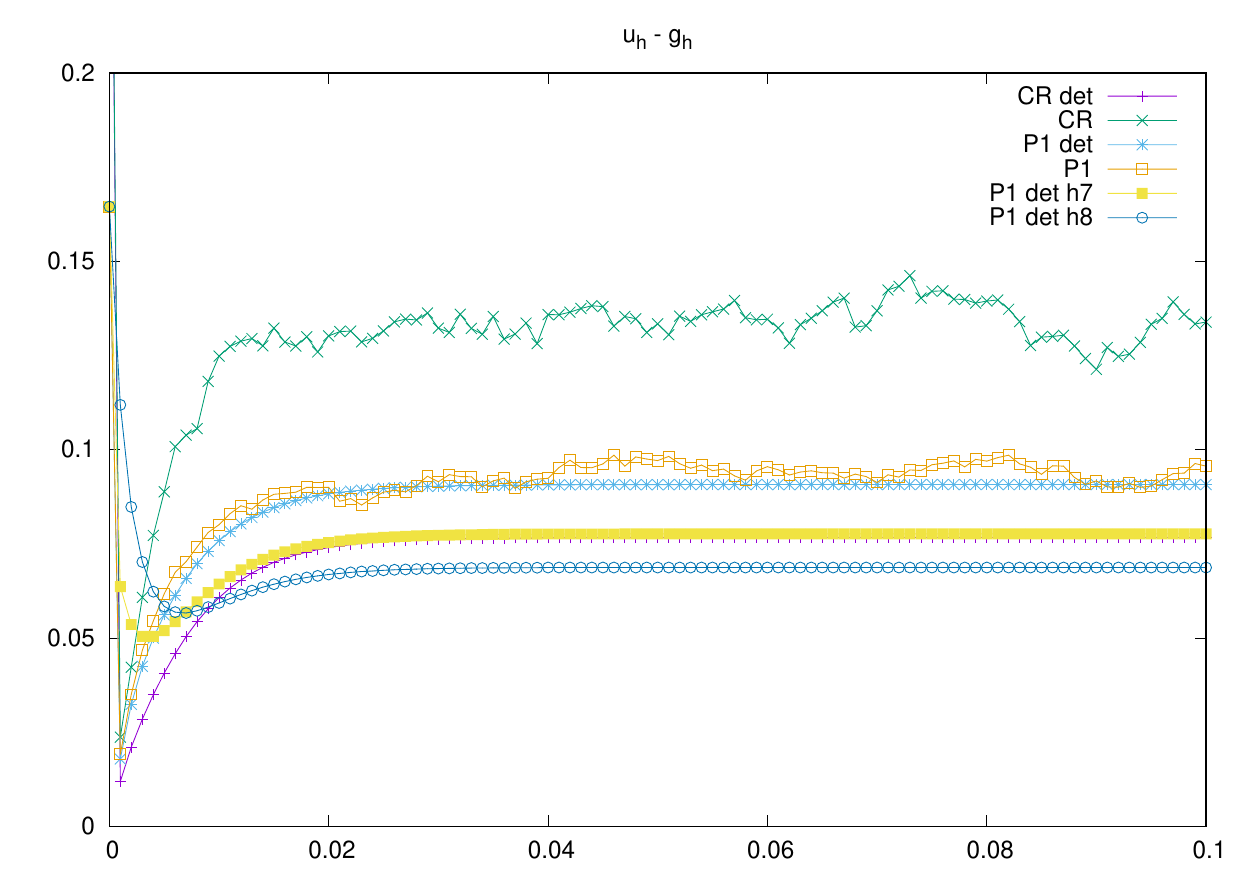}
\caption{Evolution of the discrete energy  (left) and evolution of the discrete error $t_i\rightarrow \frac{\lambda}{2}\|X_{\varepsilon,h}^i - \tilde{g}_h\|^2$ (right).
}
\label{fig_ener}
\end{figure}

\begin{figure}[!htp]
\center
\includegraphics[width=0.49\textwidth]{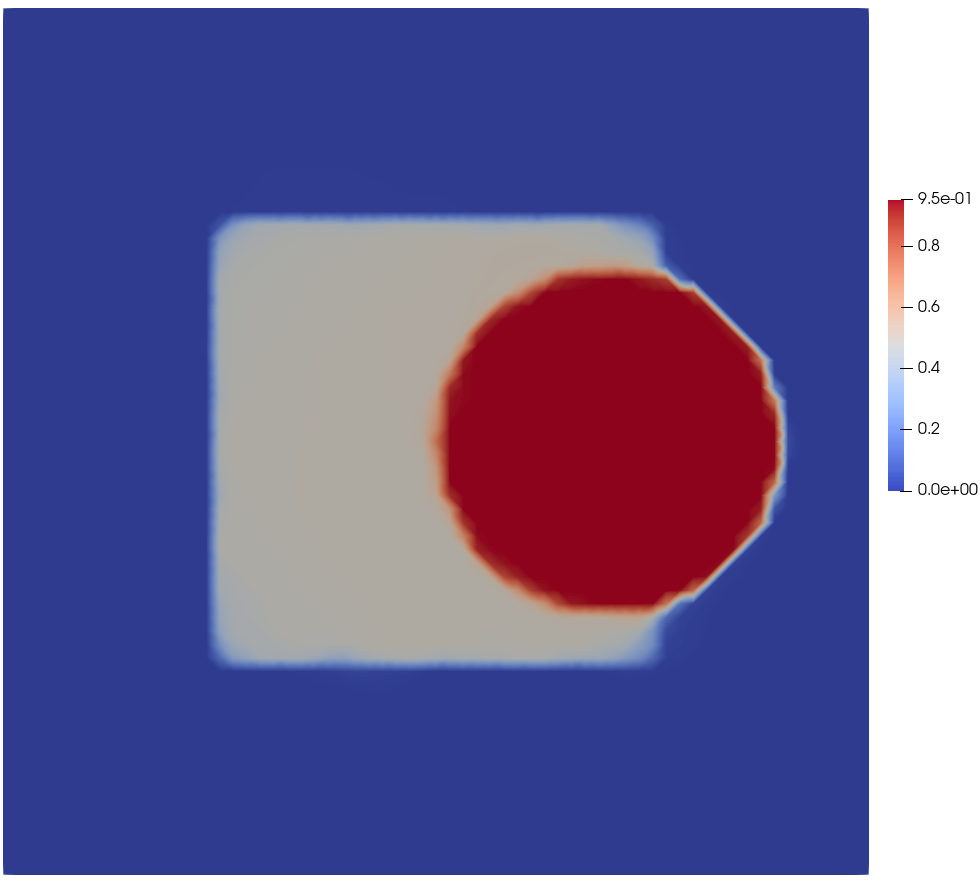}
\includegraphics[width=0.49\textwidth]{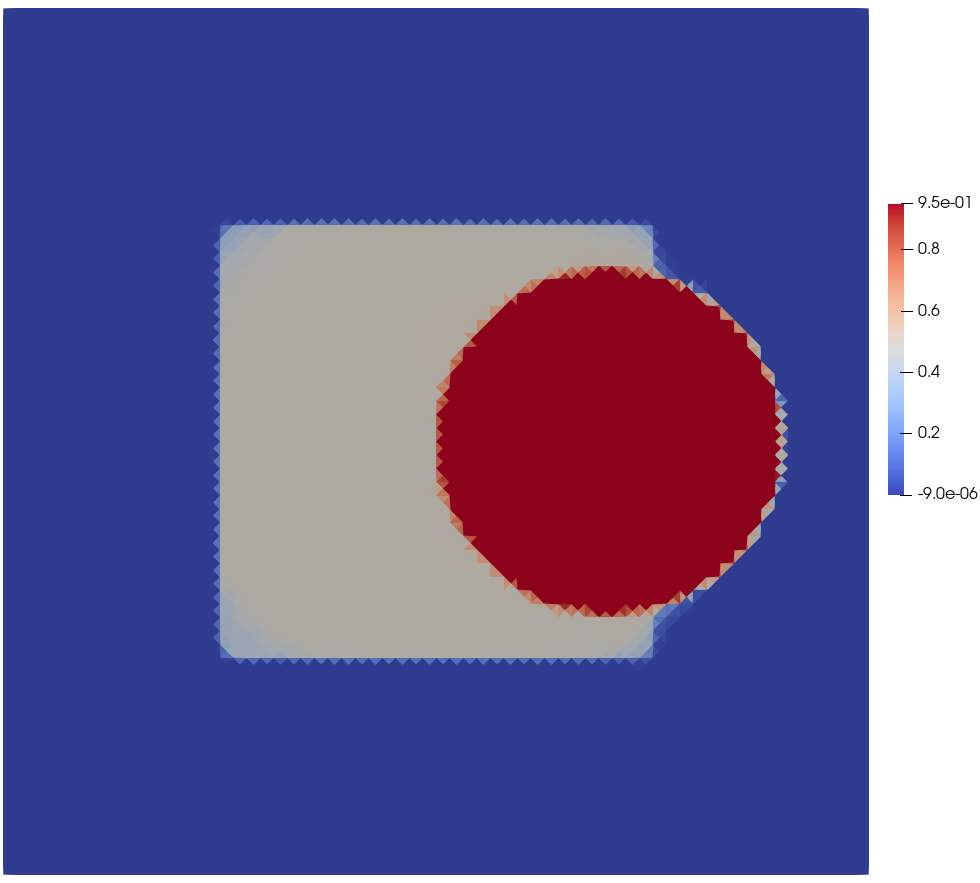}
\caption{Solution computed with the conforming finite element scheme (left) and the projected solution of the non-conforming finite element scheme (right).
}
\label{fig_finalu}
\end{figure}

\begin{figure}[!htp]
\center
\includegraphics[width=0.49\textwidth]{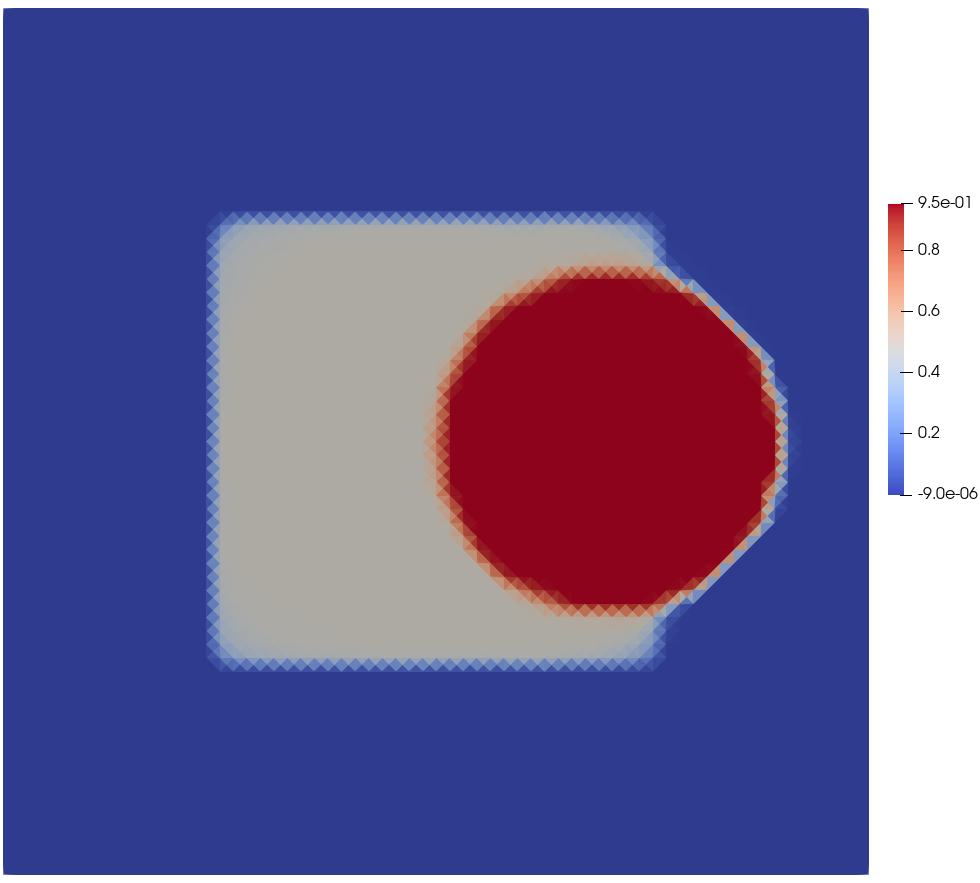}
\includegraphics[width=0.49\textwidth]{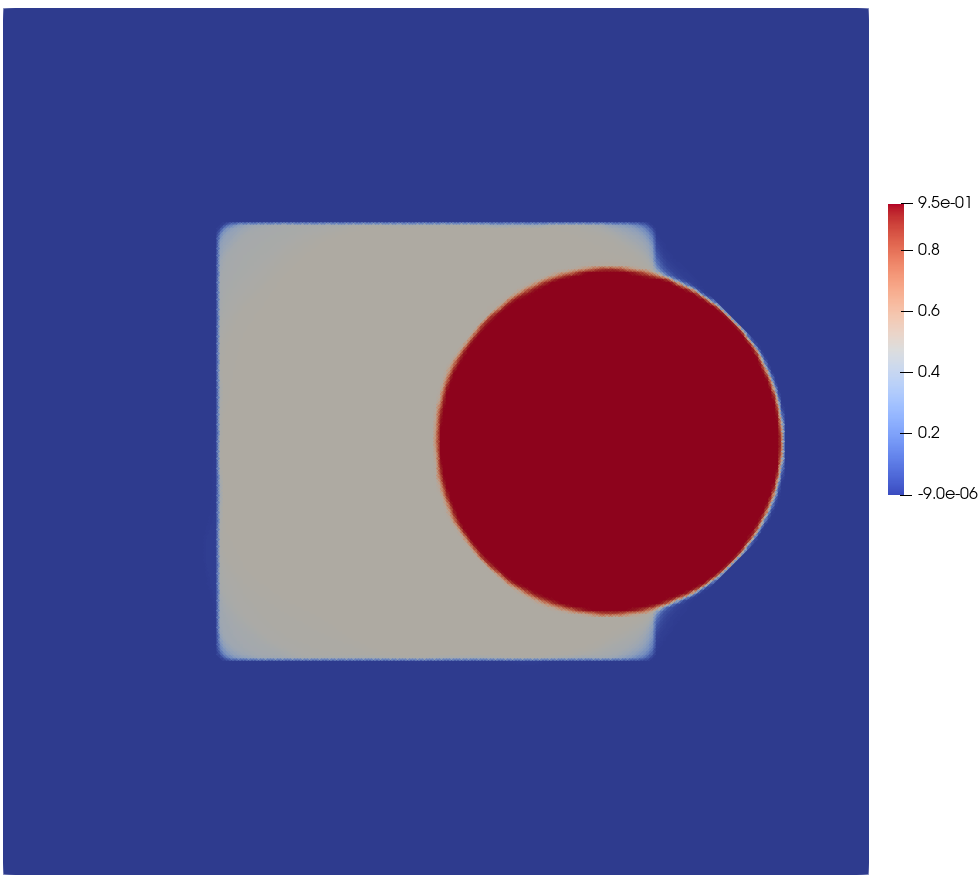}
\caption{Projected solution of the conforming finite element scheme with $\sigma=0$ for $h=2^{-6}$ (left) and $h=2^{-8}$ (right) at $T=0.1$.}
\label{fig_finalu_pw}
\end{figure}

\appendix

{

\section{Proofs of the results from Section \ref{SLCS}}

\subsection{Proof of Theorem \ref{comp_char}}\label{proof_comp_char}

The proof is analogous to that of the generalized Arzela-Ascoli theorem e.g. \cite[Theorem 7.6]{kelley}. Denote by $\tau_{\mathbf p}$ the topology of pointwise convergence on $Y^{[0,T]}$. Apparently, $\tau_{\mathbf p}\subseteq\tau_{\mathbf u}$. Basically, (i) yields that $\overline{M}^{\tau_{\mathbf p}}$ is compact in $Y^{[0,T]}$ by the Tychonoff theorem, and the traces of $\tau_{\mathbf p}$ and $\tau_{\mathbf u}$ coincide on $\overline{M}^{\tau_{\mathbf p}}$ by (ii). To see the latter, fix an absolutely convex neighbourhood of zero $O$ and get $\delta>0$ and $m\in\Bbb N$ from (ii). Let $D$ be a finite subset of $(T\Bbb Q)\cap[0,T]$ that contains all $t^n_j$ for $0\le j\le n\le m$, let $D$ intersect each non-empty intersection $(t^k_{i-1},t^k_i)\cap (t^l_{j-1},t^l_j)$ whenever $1\le i\le k\le m$, $1\le j\le l\le m$, and let $D$ be a $\delta$-net in $(t^k_{i-1},t^k_i)$ for every $1\le i\le k\le m$. With these preparations, if $f,g\in M$ are such that $f(r)-g(r)\in O$ for every $r\in D$ then $f(t)-g(t)\in 3O$ for every $t\in[0,T]$. Thus, if $f,g\in\overline{M}^{\tau_{\mathbf p}}$ are such that $f(r)-g(r)\in O$ for every $r\in D$ then $f(t)-g(t)\in 3\overline{O}$ for every $t\in[0,T]$. In particular, (ii) yields that $\tau_{\mathbf p}$ is stronger than $\tau_{\mathbf u}$ on $\overline M^{\tau_{\mathbf p}}$. But since $\tau_{\mathbf p}$ is weaker than $\tau_{\mathbf u}$, the topologies coincide on $\overline M^{\tau_{\mathbf p}}$. Now (ii) also yields 
\begin{equation}\label{seqinclinfty}
\overline M^{\tau_{\mathbf p}}=\bigcap_{n=1}^\infty\left\{\overline {M_n^\uparrow}^{\tau_{\mathbf p}}\cup\bigcup_{m=1}^{n-1}\overline {M\cap Q_m}^{\tau_{\mathbf p}}\right\}\subseteq Q_\infty\cup\bigcap_{n=1}^\infty\overline {M_n^\uparrow}^{\tau_{\mathbf p}}\subseteq Q_\infty\cup C([0,T];Y)=Q_c.
\end{equation}

The implication (iii) $\Rightarrow$ (i) is obvious and one gets (iii) $\Rightarrow$ (ii) by contradiction.

To prove (iii) $\Rightarrow$ (iv) and the assertion in Remark \ref{rem_4}, we are going to use only the fact that $f(s+)$ and $f(t-)$ exist for every $0\le s<t\le T$ and every $f$ in $\overline M$. For let $K$ be the closure of $M$ and define
$$
R=\left\{f(t-),f(t),f(t+):\,t\in[0,T],\,f\in K\right\}
$$
where $f(0-):=f(0)$ and $f(T_+):=f(T)$. The definition of $R$ is correct since we know by \eqref{seqinclinfty} that $K\subseteq Q_\infty\cup C([0,T];Y)$ if (iii) holds, or we refer to Remark \ref{rem_3_seq_comp}. Let us prove that $R$ is compact in $Y$. For let $\mathcal U$ be an ultrafilter in $R$ and define
$$
S_U=\{(t,f)\in[0,T]\times C:\,\{f(t-),f(t),f(t+)\}\cap U\ne\emptyset\}.
$$
Then $\{S_U:\,U\in\mathcal U\}$ is a basis of a filter in the compact space $[0,T]\times K$, and therefore it converges to some $(s,g)\in [0,T]\times K$. We conclude that
$$
\left[\left(g(s-)+O\right)\cup\left(g(s)+O\right)\cup\left(g(s+)+O\right)\right]\cap U\ne\emptyset
$$
holds for every $U\in\mathcal U$ and every neighbourhood $O$ of zero in $Y$. Since $\mathcal U$ is an ultrafilter,
$$
\left[\left(g(s-)+O\right)\cup\left(g(s)+O\right)\cup\left(g(s+)+O\right)\right]\cap R\in\mathcal U
$$
and so $\mathcal U$ converges to one of the elements in the set $\{g(s-),g(s),g(s+)\}$.

\subsection{Proof of Corollary \ref{cor_1_me}}\label{proof_cor_1_me}

Say that $f$ takes values in some compact $K$ for every $f\in M$, let $\{[|\cdot|_n<1]:\,n\in\Bbb N\}$ be a basis of absolutely convex open neighbourhoods of zero in the compact set 
$$
C=\bigcup_{\max\,\{|a|,|b|\}\le 1}(aK+bK)
$$
for some continuous pseudonorms $|\cdot|_n$ on $Y$ and define
$$
d(y_1,y_2)=\sum_{n=1}^\infty 2^{-n}\min\,\{1,|y_1-y_2|_n\},\qquad y_1,y_2\in Y.
$$
Then
\begin{equation}\label{metric_D}
D(f,g)=\sup\,\{d(f(t),g(t)):\,t\in[0,T]\},\qquad f,g\in Q([0,T];Y)
\end{equation}
metrizes the topology on $M$.

\subsection{Proof of Corollary \ref{cor_2_me}}\label{proof_cor_2_me}

It suffices to prove the assertion for compact sets $M$ in $Q_c([0,T];Y)$. The mapping $f\mapsto D(f,g)$ is $\mathcal Y_T$-measurable for every $g\in Q([0,T];Y)$ by Remark \ref{rem_1_ct}, hence the traces of $\mathscr B(Q([0,T];Y))$ and $\mathcal Y_T$ coincide on $M$ as $(M,D)$ is a separable metric space by Corollary \ref{cor_1_me}. Now it suffices to prove that $M$ itself belongs to $\mathcal Y_T$. According to Theorem \ref{comp_char}, there exist $\{m_n:\,n\in\Bbb N\}\subseteq\Bbb N$ and $\{\delta_n:\,n\in\Bbb N\}\subseteq(0,\infty)$ such that $M\subseteq R$ where
$$
R=\left[\bigcap_{t\in [0,T]}\pi_t^{-1}[K]\right]\cap\bigcap_{n=1}^\infty\left\{\left[\bigcup_{j=1}^{m_n}Q_j\right]\cup\left[\bigcap_{|t-s|\le\delta_n}\{f:\,|f(t)-f(s)|_n\le 1\}\right]\right\},
$$
and $K$ and $\{|\cdot|_n\}$ are the same as in the proof of Corollary \ref{cor_1_me}. But $R$ is closed (as an intersection of closed sets), relatively compact in $Q_c([0,T];Y)$ by Theorem \ref{comp_char} (hence compact), and $\mathcal Y_T$-measurable as
$$
R=\left[\bigcap_{t\in D_T}\pi_t^{-1}[K]\right]\cap\bigcap_{n=1}^\infty\left\{\left[\bigcup_{j=1}^{m_n}Q_j\right]\cup\left[\bigcap_{t,s\in D_T,\,|t-s|\le\delta_n}\{f:\,|f(t)-f(s)|_n\le 1\}\right]\right\}
$$
where $D_T=(T\Bbb Q)\cap[0,T]$. Thus the trace of $\mathscr B(Q([0,T];Y))$ on $R$ is a subset of $\mathcal Y_T$ and, in particular, $M\in\mathcal Y_T$.

\subsection{Proof of Proposition \ref{portmanteau}}\label{proof_portmanteau}

It suffices to prove the first assertion for $F_n$ and $F$ real-valued (otherwise compose theses functions with $x\mapsto\min\,\{x,m\}$ and then let $m\to\infty$). If $t\in(0,\infty)$ then set $R=(-\infty,t]$, and we have, for every $r\in(0,1)$,
$$
\mu_n(F_n\in R)\le r+\mu_n\left(\overline{\bigcup_{k=m}^\infty[F_k\in R]\cap K_{r,k}}\right),\qquad m\le n,
$$
so
\begin{align*}
\limsup\mu_n(F_n\in R)&\le r+\mu\left(\bigcap_{m=1}^\infty\overline{\bigcup_{k=m}^\infty[F_k\in R]\cap K_{r,k}}\right)
\\
&\le r+\mu(F\in R)+\mu^*(D_r)
\end{align*}
by the classical Portmanteau theorem, cf. \cite[Corollary 8.2.10]{bog}, hence
$$
\liminf\mu_n(F_n>t)\ge\mu(F>t)
$$
and therefore
$$
\int_XF\dd\mu=\int_0^\infty\mu(F>t)\,\dd t\le\liminf\int_0^\infty\mu_n(F_n>t)\,\dd t=\liminf\int_XF_n\,\dd\mu
$$
by the Fatou lemma. The second part of the proof is analogous but we take any closed set $R$. In this way, we get
$$
\limsup\mu_n(F_n\in R)\le\mu(F\in R)
$$
for every $R$ closed, therefore $\limsup\mu_n(F_n\in\cdot\,)\Rightarrow\mu(F\in\cdot\,)$. The first part of the proof now yields that $|F|$ is integrable with respect to $\mu$, and we get the claim by the assumption of uniform integrability of $|F_n|\,d\mu_n$.

\section{Bounded variation spaces}

\begin{lemma}\label{lem_lsc} The functional
$$
\mathcal I(u)=\sup\,\left\{\int_{\mathcal O}u\operatorname{div}\varphi\,\dd x:\,\varphi\in C^\infty(\Bbb R^d;\Bbb R^d),\,|\varphi|\le 1\right\},\qquad u\in\Bbb L^1
$$
satisfies
$$
\mathcal I(u)=\|\nabla u\|_{\operatorname{TV}(\mathcal{O})}+\int_{\partial\mathcal O}{|u|}\,\dd x,\qquad \mathrm{for}\, u\in BV(O)
$$
and $\mathcal I(u)=\infty$ for $u\in\Bbb L^1\setminus BV(O)$. In particular, $\mathcal I$ is lower semicontinuous on $(\Bbb L^1,\operatorname{weak})$ and convex on $BV(\mathcal{O})$ and $\mathcal J$ is lower weakly semicontinuous on $\Bbb L^2$ and convex on $\Bbb L^2\cap BV(\mathcal O)$.
\end{lemma}

\begin{proof}
If $\mathcal I(u)<\infty$ then $u\in BV(O)$ e.g. by Proposition 3.6 in \cite{AFP}. If $u\in BV(O)$ then
$$
\int_{\mathcal O}u\operatorname{div}\varphi\,\dd x=\int_{\partial\mathcal O}u(\varphi,\nu)\,\dd S-\int_{\mathcal O}\varphi\cdot\,\dd\,\nabla u,\qquad\varphi\in C^\infty(\Bbb R^d;\Bbb R^d)
$$
where $\nu$ is the outer normal vector field on $\partial O$ by the integration by parts formula, see e.g. (3.85) in \cite{AFP}, so
$$
\mathcal I(u)=\sup\,\left\{\int_{\overline{\mathcal O}}\varphi\cdot\,\dd\,\theta:\,\varphi:\Bbb R^d\to\Bbb R^d\text{ Borel measurable},\,|\varphi|\le 1\right\}\,,
$$
by a standard density argument where $\theta=u\nu\mathcal H_{d-1}|_{\partial O}-\nabla u$. Hence
$$
\mathcal I(u)=\|\theta\|_{\operatorname{TV}(\overline{\mathcal O})}=\|\nabla u\|_{\operatorname{TV}(\mathcal O)}+\|u\nu\mathcal H_{d-1}\|_{\operatorname{TV}(\partial\mathcal O)}
=\|\nabla u\|_{\operatorname{TV}(\mathcal O)}+\int_{\partial\mathcal O}|u|\,\dd S.
$$
\end{proof}

\begin{remark}\label{funct_i_sup} There exists a countable subset $\mathcal H$ of $C^\infty(\Bbb R^d)$ such that
$$
\mathcal I(u)=\sup\,\left\{\int_{\mathcal O}u\phi\,\dd x:\,\phi\in\mathcal H\right\},\qquad u\in\Bbb L^1
$$
by separability of $\{\operatorname{div}\varphi:\,\varphi\in C^\infty(\Bbb R^d;\Bbb R^d),\,|\varphi|\le 1\}$ in $C^\infty(\Bbb R^d)$. 
\end{remark}

\section{Besov spaces}

\begin{lemma}\label{besov_discr} Let $Y$ be a Banach space and let $f:[0,T]\to Y$ be a continuous function linear on every $[t_i,t_{i+1}]$ for $i=0,\dots,N-1$ and define
$$
f_{i,a}=\left[\tau\sum_{j=i}^N\|f(t_{j})-f(t_{j-i})\|^a\right]^\frac{1}{a},\qquad f_{i,\infty}=\max_{i\le j\le N}\|f(t_{j})-f(t_{j-i})\|.
$$
Then
$$
\|f\|_{L^r(0,T)}\le\left[\sum_{i=0}^N\tau\|f(t_i)\|^r\right]^\frac{1}{r},\qquad\|f\|_{L^\infty(0,T)}\le\,\max_{0\le i\le N}\|f(t_i)\|
$$
$$
\quad [f]_{B^s_{p,q}}\le\frac{8}{s(1-s)}\left(\sum_{i=1}^{N-1}\tau\frac{f_{i,p}^q}{t_i^{1+sq}}\right)^\frac{1}{q},\quad [f]_{B^s_{p,\infty}}\le 3\max_{1\le i<N}\frac{f_{i,p}}{t_i^{s}}
$$
for every $s\in(0,1)$, $p\in[1,\infty]$ and $r,q\in[1,\infty)$.
\end{lemma}

}

\bibliographystyle{plain}
\bibliography{references}

\end{document}